\documentclass[11pt,reqno, a4paper]{amsart}
\usepackage{amsfonts}
\usepackage{amssymb}
\usepackage{txfonts}
\usepackage{mathrsfs}
\usepackage{bbm}
\usepackage{enumerate}

\newtheorem{thm}{Theorem}[section]
\newtheorem{lem}[thm]{Lemma}
\newtheorem{prop}[thm]{Proposition}
\newtheorem{cor}[thm]{Corollary}

\newtheorem{rmk}[thm]{Remark}
\numberwithin{equation}{section}

\newcommand{\vpi}{\varphi}
\newcommand{\p}{\partial}

\newcommand{\reff}[1]{(\ref{#1})}
\newcommand{\norm}[1]{\left\Vert#1\right\Vert}

\newcommand{\com}[1]{\big[#1\big]}
\newcommand{\set}[1]{\left\{#1\right\}}
\newcommand{\inner}[1]{\left(#1\right)}

\newcommand{\abs}[1]{\left\vert#1\right\vert}

\newcommand\cs{{\mathcal S}}
\newcommand\cm{{\mathcal M}}
\newcommand\cl{{\mathcal L}}

\newcommand{\V}{\,\,\big|\,\,\,}
\newcommand\nn{{\mathbb N}}
\newcommand\rr{{{\mathbb R}}}
\newcommand\zz{{\mathbb Z}}

\begin{document}

\title[Gevrey regularity]{Gevrey regularity of subelliptic Monge-Amp\`{e}re
equations in the plane}

\author[H. Chen, W.-X. Li, and C.-J. Xu]{Hua CHEN, Wei-Xi LI \and
Chao-Jiang XU}

\address{School of Mathematics and Statistics, Wuhan University,
Wuhan 430072, China} \email{chenhua@whu.edu.cn}

\address{School of Mathematics and Statistics, Wuhan University,
Wuhan 430072, China} \email{wei-xi.li@whu.edu.cn}

\address{ School of Mathematics and Statistics, Wuhan University, Wuhan
430072, China
\newline\indent and
\newline\indent
Universit\'e de Rouen, UMR6085-CNRS-Math\'ematiques,
Avenue de l'Universit\'e, BP.12,
\newline\indent
76801 Saint Etienne du Rouvray, France} \email{Chao-Jiang.Xu@univ-rouen.fr}

\thanks{This work is partially  supported by the NSFC}

\date{September 09 2009}


\begin{abstract}
  In this paper, we establish the Gevrey regularity of solutions for a class of
  degenerate Monge-Amp\`{e}re equations in the plane, under the
  assumption that one principle entry of the Hessian is strictly positive
  and an appropriately finite type degeneracy.
\end{abstract}

\maketitle

\section{Introduction}

In this paper, we study the regularity problem for the real
Monge-Amp\`{e}re equation
\begin{equation}\label{MA-0}
\det \left(D^2 u \right)=k(x),\quad x\in\Omega\subset\rr^d,
\end{equation}
where $\Omega$ is an open domain of $\rr^d, d\geq 2$. We consider
the convex solution $u$ of equation \reff{MA-0}, then $k$ is a
nonnegative function. In the case when $k>0,$ the equation
\reff{MA} is elliptic and the theory is well developed. For
instance, it's shown in \cite{MR739925} that there exists a unique
solution $u$ to the Dirichlet problem for \reff{MA-0}, smooth up to
the boundary of $\Omega,$ provided that $k$ is smooth and the
boundary $\p\Omega$ of $\Omega$ is strictly convex. In the
degenerate case, i.e.,
$$
\Sigma_k=\left\{x\in \Omega; k(x)=0, \nabla
k(x)=0\right\}\not=\emptyset.
$$
The equation \reff{MA-0} is then a full nonlinear degenerate
elliptic equation. The existence and uniqueness of the solution for
the Dirichlet problem of the equation \reff{MA-0} have already been
studied in \cite{MR1687172}. Also in \cite{HZ}, they proved that the
Monge-Amp\`{e}re equation has a $C^\infty$ convex local solution if
the order of degenerate point for the smooth coefficient $k$ is
finite.

As far as the regularity problem is concerned, a result in \cite{Z}
proved that, for the degenerate Monge-Amp\`{e}re equation, if the
solution $u\in C^\rho$ for $\rho>4$ (so that it is a classical
solution), then $u$ will be $C^\infty$ smooth.

However, in general, the convex solution $u$ to \reff{MA-0} is at
most in $C^{1,1}$ if $k$ is only smooth and nonnegative (see \cite
{MR1430436} for example). To get a higher regularity, some extra
assumptions are needed to impose on $k.$ This problem has been
studied by P. Guan \cite{MR1488238} in two dimension case, in which
the smoothness of $C^\infty$ for a $C^{1,1}$ solution $u$ of the
equation \reff{MA-0} is obtained, if $k$ vanishes in finite order,
i.e. $k\approx x^{2\ell}+A y^{2n}$ with $\ell\leq n, A\geq0$, and
one principal curvature of $u$ is strictly positive. In a recent
paper \cite{guan-sawyer1}, the last assumption is relaxed to the
bounding of trace of Hessian from below, i.e., $\triangle u\geq
c_0>0$. For such $C^\infty$ regularity problem, see also earlier
work of C.-J. Xu \cite{MR864418} which is concerned with the
$C^\infty$ regularity for general two-dimensional degenerate
elliptic equation. In a recent work \cite{MR2137289}, the authors
extended Guan's two-dimensional result of \cite{MR1488238} to higher
dimensional case.

It is natural to ask that, in the degenerate case, would it be the
best possible for the regularity of solution here to be $C^\infty$
smooth? One may expect that, in case of coefficient $k$ with higher
regularity, the solution $u$ would have better regularity than
$C^\infty$ smooth. we will introduce the Gevrey class, an
intermediate space between the spaces of the analytic functions and
the $C^\infty$ functions. There is well-developed theory on the
Gevrey regularity (see the definition later) for nonlinear elliptic
equations of any order, see \cite{Friedman58} for instance. For the
linear degenerate elliptic problem, there have been many works on
the Gevrey hypoellipticity of linear subelliptic operators of second
order (e.g. \cite{DerridjZuily73-2,Durand78} and the reference
therein). The difficulty concerned with equation \reff{MA-0} lies on
the mixture of degeneracy and nonlinearity.

In this paper, we attempt to explore the regularity of solutions of
equation \reff{MA-0} in the frame of Gevrey class. We study the
problem in two dimension case
\begin{equation}\label{MA}
 u_{xx} u_{yy}-u_{xy}^2=k(x,y),\quad (x,y)\in\Omega,
\end{equation}
and assume that $u_{yy}>0$, then we can apply the classic partial
Legendre transformation (see \cite{MR1079936} for instance), to
translate the equation \reff{MA} to the following divergence form
quasi-linear equation
\begin{equation}\label{pre}
   \p_{ss}w(s,t)+\p_t\set{k(s,w(s,t))\p_t w(s,t)}=0.
\end{equation}
This quasi-linearity allows us to adopt the idea used in
\cite{CLXJ}, to obtain the Gevrey regularity for the above
divergence form equation. In order to go back to the original
problem, i.e., the Gevrey regularity for the equation \reff{MA}, a
key point would be to show that the Gevrey regularity is invariant
under the partial Legendre transformation, which will be proved in
Section \ref{0811032}.

Now let us recall the definition of the space of Gevrey class
functions, which is denoted by $G^\sigma(U)$, for $\sigma\geq
1,$ with $U$ an open subset of $\mathbb{R}^d$ and $\sigma$
being called Gevrey index. We say that $f\in G^{\sigma}(U)$
if $f\in C^\infty(U)$ and for any compact subset $K$ of
$U$, there exists a constant $C_K$, depending only on $K$,
such that for all multi-indices $\alpha\in\mathbb{Z}_+^d$,
\begin{equation*}
\|\partial^\alpha{f}\|_{L^\infty(K)} \leq C_K^{|\alpha|+1}(\abs\alpha!)^{\sigma}.
\end{equation*}
The constant $C_K$ here is called the Gevrey constant of $f.$  We
remark that the above inequality is equivalent to the following
condition:
$$
\|\partial^\alpha{f}\|_{L^2(K)}\leq C_K^{|\alpha|+1}(|\alpha|!)^{{\sigma}}.
$$
In this paper, both estimates above will be used. Observe that
$G^1(U)$ is the space of real analytic functions in $U$.

We state now our main result as follows, where $\Omega$ is an open
neighborhood of origin in $\rr^2$.

\begin{thm}\label{main}
  Let $u$ be a $C^{1,1}$ weak convex solution to the Monge-Amp\`{e}re
  equation \reff{MA}. Suppose that  $u_{yy}\geq c_0>0$ in
  $\Omega$ and that  $k(x,y)$ is a smooth
  function defined in $\Omega$, satisfying
  \begin{equation}\label{sub+}
  c^{-1}(x^{2\ell}+A\, y^{2n})\leq k(x, y)\leq c( x^{2\ell}+ A\,
  y^{2n}),\quad\quad (x, y)\in \Omega
  \end{equation}
  where $c>0, A\geq 0$ and $\ell\leq n$ are two nonnegative
  integers. Then $u\in G^{\ell+1}(\Omega),$
  provided $k\in G^{\ell+1}(\Omega).$
\end{thm}

\begin{rmk}
  If $k$ is $C^\infty$ smooth and satisfies the condition \reff{sub+}, and
  $u_{y y}>0$, P. Guan \cite{MR1488238} has proved that a $C^{1,1}$
   solution of the equation \reff{MA} will be $C^\infty$ smooth.
  In \cite{guan-sawyer1}, the assumption {that $u_{yy}>0$}
  is relaxed to the bounding
  of trace of Hessian from below, i.e., $\triangle u\geq c_0>0$, but the assumption
  \reff{sub+} is changed to $A>0$ and $\ell=n$. Our main contribution here
  is to obtain the Gevrey regularity $G^{\ell+1}$.
\end{rmk}

\begin{rmk}
  The regularity result of main theorem seems the best possible,
  since in the particular case of $\ell=0,$
  we have $u\in G^1(\Omega)$ (i.e., the solution is analytic in $\Omega$),
  and in this case the equation \reff{MA} is elliptic, thus our result
  coincides with the well-known analytic regularity result for nonlinear
  elliptic equations. We can also {justify that} if $k$ is
  independent of second variable $y$ (then $A=0$), the equation \reff{pre} is linear,
  it is known that, see \cite{DerridjZuily73-2}, the optimal regularity
  result is that { the solution lies} in $G^{\ell+1}$.
\end{rmk}

\begin{rmk}
The extension of above result to higher dimensional cases and more
general models of the Monge-Am\`{e}re equations with $k=k(x,u,Du)$
is our coming work. By using the results of \cite{MR2137289}, the
idea is the same.
\end{rmk}

The paper is organized as follows: the section \ref{0811031} is
devoted to proving the Gevrey regularity for the quasi-linear
equation \reff{pre}. In Section \ref{0811032} we prove our main
result by virtue of the classic partial Legendre transformation. We
prove the technical  lemmas in Section \ref{lemmas}.


\section{Gevrey regularity of quasi-linear subelliptic equations}
\label{0811031}

In this section we study the Gevrey regularity of solutions for the
following quasi-linear equation near the origin of $\rr^2$
\begin{equation}\label{quasi}
  \p_{ss}w+\p_t\inner{k(s,w)\p_tw}=0.
\end{equation}
We assume that $k(s,w)$ satisfies the condition
\begin{equation}\label{0810271}
 c^{-1}(s^{2\ell}+ A w^{2n})\leq k(s,w)\leq c( s^{2\ell}+ A w^{2n}),
\end{equation}
where $c>1, A\geq 0$ are two constants and $\ell\leq n$ are two
positive integers. Since Gevrey regularity  is a local property, we
study the problem on the unit ball in $\rr^2$,
\[
B=\set{(s,t)\V s^2+t^2<1},
\]
and denote  $W=[-1, 1]\times [-\|w\|_{L^\infty(\bar B)},
\|w\|_{L^\infty(\bar B)}]$. We prove the the following result in
this section.

\begin{thm}\label{08111508}
Suppose that $w(s,t)\in C^\infty(\bar B)$ is a solution to the
quasi-linear equation \reff{quasi}, and that $k\in G^{\ell+1}(W).$
Then $w\in G^{\ell+1}(B).$
\end{thm}

We recall some notations and elementary results for the Sobolev
space and pseudo-differential operators. Let $H^\kappa(\mathbb
R^{2}),\kappa\in\mathbb R,$ be the classical Sobolev space equipped
with the norm $\|\cdot\|_\kappa$. Observe
$\norm{\cdot}_0=\norm{\cdot}_{L^2(\rr^2)}$.  Recall that $
H^\kappa(\rr^2)$ is an algebra if $\kappa>1$. We need also the
interpolation inequality for Sobolev space: for any $\varepsilon>0$
and any $r_1< r_2<r_3,$
\begin{equation}\label{inter}
\|h\|_{r_2}\leq \varepsilon \|h\|_{r_3}+
\varepsilon^{-(r_2-r_1)/(r_3-r_2)}\|h\|_{r_1},\quad\forall ~h\in
H^{r_3}(\rr^2).
\end{equation}

Let $U$ be an open subset of $\rr^{2}$ and $S^a(U), a\in \rr,$ be
the symbol space of classical pseudo-differential operators. We
say $P=P(s,t,D_s,D_t)\in{\rm Op}(S^a(U))$,
 a pseudo-differential operator of
order $a,$ if its symbol $\sigma(P)(s,t;\zeta,\eta)\in S^a(U)$
 with $(\zeta,\eta)$ the dual variable of $(s,t).$  If $P\in{\rm
Op}(S^a(U))$, then $P$ is a continuous operator from
$H_{c}^\kappa(U)$ to $H_{loc}^{\kappa-a}(U)$. Here
$H_{c}^\kappa(U)$ is the subspace of $H^\kappa(\rr^2)$ consisting
of  the distributions having their compact support in $U$, and
$H_{loc}^{\kappa-a}(U)$ consists of the distributions $h$ such
that $\phi h\in H^{\kappa-a}(\rr^2)$ for any $\phi\in
C_0^\infty(U)$. For more detail on the pseudo-differential
operator, we refer to the book \cite{Treves80}. Remark that if
$P_1\in {\rm Op}(S^{a_1})$, $P_2\in {\rm Op}(S^{a_2}(U))$, then
$[P_1,~P_2]\in {\rm Op}(S^{a_1+a_2-1}(U)).$ In this paper, we
shall use the pseudo-differential operator
$\Lambda^r=\inner{1+\abs{D_s}^2+\abs{D_t}^2}^{\frac r 2}$ of order
$r, r\in\rr,$ whose symbol is given by
\[
\sigma(\Lambda^r)=\inner{1+\zeta^2+\eta^2}^{\frac r 2}.
\]
In the following discussions, we denote, for {$P\in {\rm
Op}(S^{a}),$}
\[
{\norm{P\partial^m
v}_{\kappa}=\sum_{\abs\alpha=m}\norm{P\partial^\alpha
v}_{\kappa}}\quad {\rm
and}~~\com{v}_{j,U}=\sum\limits_{\abs\gamma=j}\norm{\p^\gamma
v}_{L^\infty(U)}.
\]

We consider the following linearized operator corresponding
to \reff{quasi} and the solution $w$,
\begin{equation*}
  \cl =\p_{ss}+\p_t \big(\tilde k (s, t) \p_t\,\, \cdot\,\,\big),
\end{equation*}
where $\tilde k(s,t)=k(s,w(s,t))$. To simplify the notation, we
extended smoothly the function $\tilde k$ to $\rr^2$ by constant
outside of $\bar B$, similar for $k$. We have firstly the following
subelliptic estimate.

\begin{lem}\label{lem3.1}
Under the assumption \reff{0810271}, for any $r\in\rr$, there
exists $C_r>0$ such that
\begin{equation}\label{08110105}
  \norm{v}^2_{r+\frac{1}{\ell+1}}+\norm {\p_s \Lambda^r v}_0^2
  +\|\tilde k^{\frac{1}{2}}\p_t\Lambda^r v\|^2_0
  \leq  C_r\set{\norm{\cl v}^2_{r-\frac{1}{\ell+1}}+\|v\|^2_0},
\end{equation}
for any $v\in C_0^\infty( B)$, where $C_r$ depends only on
$\com{\tilde k}_{j, \bar B}, 0\leq j\leq 2$.
\end{lem}

\begin{rmk}
By using { Fa\`{a} di Bruno's} formula, $\com{\tilde k}_{j, \bar B}$
is bounded by a polynomial of $\com{k}_{i, W}, \com{w}_{i, \bar B},
0\leq i\leq j$.
\end{rmk}

\begin{proof} Firstly, we study the case of $r=0$. Observe
\begin{equation}\label{sub0}
  \norm {\p_s v}_0^2+\|\tilde k^{\frac{1}{2}}\p_t v\|_0^2
=\norm {\p_s v}_0^2+\int_{\rr^2}\tilde k (s, t) |\p_t v(s, t)|^2ds
dt =-\inner{ \cl v, v}.
\end{equation}
Then the assumption \reff{0810271} implies
\begin{equation*}
\norm {\p_s v}_0^2+\|s^{\ell}\p_t v\|_0^2 \leq c \big\{\norm {\p_s
v}_0^2+\|\tilde k^{\frac{1}{2}}\p_t v\|_0^2\big\} =-c\inner{ \cl
v, v}.
\end{equation*}
Since the vector fields $\{\p_s, s^{\ell}\p_t\}$ satisfies the
H\"ormander's condition of order $\ell$, we get (see \cite{MR730094,rotschild-stein})
\begin{equation}\label{sub1}
\norm {v}_{\frac{1}{\ell+1}}^2\leq C_0 \big\{\norm {\p_s
v}_0^2+\|\tilde k^{\frac{1}{2}}\p_t v\|_0^2+\|v\|_0^2\big\}
=-C_0\inner{ \cl v, v}+C_0\|v\|_0^2.
\end{equation}
By Cauchy-Schwarz inequality, we have proved \reff{08110105} with
$r=0$. Since we have extended $\tilde k$ to $\rr^2$, \reff{sub0}
\reff{sub1} also hold  for any $v\in\cs(\rr^2)$.

Now for the general case, we have
\begin{align*}
  &\|\p_s\Lambda^r v\|_0^2+\|\tilde k^{\frac{1}{2}}\p_t\Lambda^r v\|_0^2
  =-\inner{ \Lambda^r \cl v, ~\Lambda^r v}-
  \inner{\com{\tilde k,~\Lambda^r} \p_t v, ~ \p_t\Lambda^r v}\\
  &\leq \|\cl v\|_{r-\frac{1}{\ell+1}}^2+\norm{ v}_{r+\frac{1}{\ell+1}}^2
  -\inner{\com{\tilde k,~\Lambda^r} \p_t v, ~ \p_t\Lambda^r v}.
\end{align*}
Since for $v\in C_0^\infty(B)$, we have $\Lambda^r v\in
\cs(\rr^2)$. Then \reff{sub1} implies
\begin{align}
  &\norm{ v}_{r+\frac{1}{\ell+1}}^2+\|\p_s\Lambda^r v\|_0^2+
  \|\tilde k^{\frac{1}{2}}\p_t\Lambda^r v\|_0^2\label{sub2}\\
  &\leq C_0 \set{\norm{\cl v}_{r-\frac{1}{\ell+1}}^2+\norm{ v}_{0}^2
  -\inner{\com{\tilde k,~\Lambda^r} \p_t v, ~ \p_t\Lambda^r v}
  }.\nonumber
\end{align}

We consider now the commutator $\com{\tilde k,~\Lambda^r}$, the
pseudo-differential  calculus give
\[
\sigma \big(\com{\tilde k,~\Lambda^r}\big)=\sum_{|\alpha|=1}\p^\alpha_{s, t} \tilde k (s,
t) \p^\alpha_{\zeta,\eta} \sigma
\big(\Lambda^r\big)(\zeta,\eta)+\sigma(R_2)(s, t, \zeta,\eta),
\]
with $\sigma(R_2)\in S^{r-2}(\rr^2)$ and
\[
\left|\inner{R_2 \p_t v, ~ \p_t\Lambda^r v}\right|\leq C_2\|v\|^2_r,
\]
where $C_2$ depends only on $\com{\,\tilde k\,}_{j, \bar B}, 0\leq
j\leq 2$. Thus
\begin{align*}
  \inner{\com{\tilde k,~\Lambda^r}\p_t v, ~ \p_t\Lambda^r v}&\leq
  C_0\norm{v}_{r}\set{\norm{\inner{\p_s\tilde k}\p_t\Lambda^{r} v}_{0}
  +\norm{\inner{\p_t\tilde k}\p_t\Lambda^{r} v}_{0}}+C\|v\|^2_r.
\end{align*}
Moreover, note that $\tilde k$ is nonnegative, and hence we have the
following well-known inequality
\begin{eqnarray}\label{090916}
|\p_s\tilde k(s,t)|^2+|\p_t\tilde k(s,t)|^2\leq 4 [\,\tilde k\,]_{2,
\rr^2} \tilde k(s,t).
\end{eqnarray}
For the sake of completeness, we will present the proof of the above
inequality later. By Cauchy-Schwarz inequality and interpolation
inequality \reff{inter}, one has
\begin{align*}
  \inner{\com{\tilde k,~\Lambda^r} \p_t v, ~ \p_t\Lambda^r v}\leq
  \frac{1}{2C_0}\Big(\norm{\tilde k^{\frac{1}{2}}\p_t\Lambda^{r}\ v}_{0}^2+
  \norm{v}_{r+\frac{1}{\ell+1}}^2\Big)+C_r\|v\|^2_0.
\end{align*}
Thus Lemma \ref{lem3.1} follows. Now it remains to show
\reff{090916}. For any  $h\in\mathbb{R},$  the following formula
holds
\[
\tilde k(s+h,t)=\tilde k(s,t)+\partial_s\tilde k (s,t)
h+\frac{1}{2}\partial_{ss}\tilde k(s_0,t)h^2, \quad s_0\in
\mathbb{R}.
\]
Observe $\tilde k\geq0,$ then  for all $h\in\mathbb{R}$ we get
$0\leq \tilde k(s,t)+\partial_s\tilde k (s,t) h+\frac{1}{2}[\,\tilde
k\,]_{2, \rr^2} h^2$. So the  the discriminant of this polynomial is
nonpositive; that is,
\[
|\p_s\tilde k(s,t)|^2\leq 2 [\,\tilde k\,]_{2, \rr^2} \tilde k(s,t).
\]
Similarly $|\p_t\tilde k(s,t)|^2\leq 2 [\,\tilde k\,]_{2, \rr^2}
\tilde k(s,t).$ This gives \reff{090916}.
\end{proof}

\begin{rmk}\label{remark2.4}
With same proof, we can also prove the following estimate
$$
  \norm{v}^2_{r+m+\frac{1}{\ell+1}}+\sum_{|\alpha|\leq m}
  \left(\norm {\p_s \Lambda^r\p^\alpha v}_0^2
  +\|\tilde k^{\frac{1}{2}}\p_t\Lambda^r\p^\alpha v\|^2_0\right)
  \leq  C_{r, m}\set{\norm{\cl v}^2_{r+m-\frac{1}{\ell+1}}+\|v\|^2_0}
$$
for any $v\in C^\infty_0(B)$.
\end{rmk}

A key technical  step in the proof of Gverey regularity is to choose
a adapted family of cutoff functions. For $0<\rho<1,$ set
\[
B_\rho=\set{(s,t)\V s^2+t^2<1-\rho}.
\]
For each integer $m\geq2$ and each number $0<\rho<1,$ we choose
the cutoff function $\vpi_{\rho, m}$ satisfying the  following
properties:
  \begin{equation}\label{0811013}
    \left\{
    \begin{array}{lll}
    {\rm supp}~
    \vpi_{\rho,m}\subset  B_{\frac{(m-1)\rho}{m}},\quad{\rm
    and~~}
    \vpi_{\rho,m}(s,t)=1{~~\rm  in~~}  B_\rho,\\[2pt]
    \sup\limits_{(s,t)\in B}\abs{\p^k \vpi_{\rho,m}}\leq
    C_k\inner{\frac{m}{\rho}}^{k}.
    \end{array}
    \right.
\end{equation}
For such cut-off functions, we have the following
\begin{lem}[Corollary 0.2.2 of \cite{Durand78}]\label{0810182}
  There exists a constant $C,$  such that
  for any $0\leq\kappa\leq 4,$ and any $f\in\mathcal{S}(\rr^{2}),$
  \begin{equation}\label{cutoffnorm}
   \norm{\inner{\p^j\vpi_{\rho, m}}f}_{\kappa}\leq
   C\set{\inner{\frac{m}{\rho}}^{j}\norm
   f_{\kappa}+\inner{\frac{m}{\rho}}^{j+\kappa}
   \norm{f}_0}, \quad 0\leq j\leq 2.
  \end{equation}
\end{lem}

We prove now Theorem \ref{08111508} by the following Proposition.

\begin{prop}\label{0810272}
  Let $w\in C^\infty(\bar B)$ be a smooth solution of the
  quasi-linear equation \reff{quasi}.
  Suppose $k\in G^{\ell+1}( \rr^2).$ Then there exists a
  constant $L,$ such that for any integer $m\geq 5$,
  we have the following estimate
  \begin{align}\label{ind}
  \begin{split}
    &\|\vpi_{\rho,m}\p^m w\|_{2+\frac{j}{\ell+1}}
    +\|\p_s\Lambda^{2+\frac{j-1}{\ell+1}}\vpi_{\rho,m}\p^m w\|_{0}
    +\|\tilde k^{\frac{1}{2}}\p_t\Lambda^{2+\frac{j-1}{\ell+1}}\vpi_{\rho,m}\p^m
    w\|_{0}\\
    &\qquad\leq \frac{L^{m-2}}{\rho^{\inner{\ell+1}(m-3)}}
    \inner{\frac{m}{\rho}}^{j}\big(\inner{m-3}!\big)^{\ell+1},
    \, 0\leq j\leq \ell+1,\, 0<\rho<1.
  \end{split}
  \end{align}
\end{prop}

\begin{rmk}
  The constant $L$ in Proposition \ref{0810272}  depends
  on $\ell, \com w_{8,\bar B}$,  the Gevrey constant of
  $k,$ and is independent of $m.$
\end{rmk}

As an immediate consequence, for each compact subset $K\subset B,$
if we choose $\rho_0=\frac{1}{2}{\rm dist}\inner{K,\p B}$.  Then $
\vpi_{\rho_0, m}=1$ on $K$ for any $m$, and \reff{ind} for $j=0$
yields,
\[
  \norm{\p^m w}_{L^2(K)}
  \leq \inner{\frac{ L}{\rho_0^{\ell+1}}}^{m+1}\inner{m!}^{\ell+1},
  \quad\forall ~m\in\nn.
\]
This gives $u\in G^{\ell+1}(B).$ The proof of Theorem \ref{08111508}
is thus completed.

\bigbreak The proof of Proposition \ref{0810272} is by induction on
$m$.  We state now the following two Lemmas, and postpone their
proof to the last section.

\begin{lem}\label{lemm2.1}
 Let $k\in G^{\ell+1}(\rr^2)$ and $w\in C^\infty(\bar B)$ be a solution of equation \reff{quasi}.
 Suppose that for some $N>5$, \reff{ind} is satisfied for any $5\leq m\leq N-1$, and that for some
 $0\leq j_0\leq \ell$, we have
  \begin{align}\label{08111306}
  \begin{split}
    &\norm{\vpi_{\rho, N}\p^N w}_{2+\frac{j_0}{\ell+1}}
    +\norm{\p_s\Lambda^{2+\frac{j_0-1}{\ell+1}}\vpi_{\rho, N}\p^N w}_{0}
    +\norm{\tilde k^{\frac{1}{2}}\p_t\Lambda^{2+\frac{j_0-1}{\ell+1}}\vpi_{\rho, N}\p^N w}_{0}\\
    &\qquad\qquad\leq \frac{C_0L^{N-3}}{{\rho}^{(\ell+1)(N-3)}}
    \inner{\frac{N}{\rho}}^{j_0}\Big(\inner{N-3}!\Big)^{\ell+1},
    \quad\forall ~0<\rho<1,
  \end{split}
  \end{align}
  where $C_0$ is a constant independent of $L, N$. Then there exists a constant $C_1$
  independent of $L, N,$ such that for any $0<\rho<1,$
  \begin{equation}\label{2.12}
    \norm{ \cl \vpi_{\rho, N}\p^N w}_{2+\frac{j_0-1}{\ell+1}}\leq
    \frac{C_1 L^{N-3}}{{\rho}^{(\ell+1)(N-3)}}
    \inner{\frac{N}{\rho}}^{j_0+1}\Big(\inner{N-3}!\Big)^{\ell+1}.
  \end{equation}
\end{lem}

\begin{lem}\label{lemm2.2}
  Let $k\in G^{\ell+1}(\rr^2)$ and $w\in C^\infty(\bar B)$ be a solution of equation \reff{quasi}.
 Suppose that for some $N>5$, \reff{ind} is satisfied for any $5\leq m\leq N-1$. Then there exists
 a constants $C_2$ independent of $L, N,$ such that
  \begin{equation}\label{08111308++}
    \norm{ \cl \vpi_{\rho, N}\p^{N-1} w}_{2+\frac{j-1}{\ell+1}}\leq
    \frac{C_2L^{N-3}}{{\rho}^{(\ell+1)(N-4)}}
    \inner{\frac{N-1}{\rho}}^{j+1}\Big(\inner{N-4}!\Big)^{\ell+1}
  \end{equation}
for all $0<\rho<1,\,\, 0\leq j\leq \ell+1.$
\end{lem}

Here and throughout the proof, $C$ and $C_j$ are used to denote
suitable constants which depend on $\ell,$ $\big[\tilde
k\big]_{0,B},  \com{w}_{8,\bar B}$  and the Gevrey constant of $k,$
but it is independent of $m$ and $L.$

\begin{proof}[{\bf Proof of Proposition \ref{0810272}}]
The proof is by induction on $m.$   Firstly by using
\reff{cutoffnorm}, the direct calculus implies, for $m=5,$ all
$0<\rho<1$ and all integers $j$ with $0\leq j\leq \ell+1,$
\begin{align*}
    \norm{\vpi_{\rho, m}\p^m w}_{2+\frac{j}{\ell+1}}
    +\norm{\p_s\Lambda^{2+\frac{j-1}{\ell+1}}\vpi_{\rho, m}\p^m w}_{0}
    +\norm{\tilde k^{\frac{1}{2}}\p_t\Lambda^{2+\frac{j-1}{\ell+1}}\vpi_{\rho, m}\p^m
    w}_{0}
    \leq \frac{M_1}{\rho^{4}}
  \end{align*}
with $M_1$ a constant depending only on $\big[\tilde k\big]_{0,B},
\com{w}_{8,\bar B}$ and the constant $C$ in \reff{cutoffnorm}.  Then
\reff{ind} obviously holds for $m\leq 5$  if we choose $L\geq M_1$.

Now we can finish the proof of Proposition \ref{0810272} by induction, for any $N>5$
$$
\mbox{{\bf Claim} :\em \reff{ind} is true for $m=N$ if it is true for all $3\leq m\leq N-1$.}
$$

We prove this claim again by induction on $j$, for $0\leq j\leq \ell+1$.

\noindent{\bf Case of $j=0$}: We apply Remark \ref{remark2.4} with
$r=2-\frac{1}{\ell+1}$, $m=1$ and $v=\vpi_{\rho, N}\p^{N-1} w\in
C^\infty_0(B)$,
 \begin{align*}
       &\|\varphi_{\rho, N}\p^N w\|^2_2
       +\|\p_s\Lambda^{2-\frac{1}{\ell+1}}\vpi_{\rho, N}\p^N w\|^2_{0}
       +\|\tilde k^{\frac{1}{2}}\p_t\Lambda^{2-\frac{1}{\ell+1}}\vpi_{\rho, N}\p^N w\|^2_{0}\\
       &\leq \|\vpi_{\rho, N}\p^{N-1} w\|^2_{2+1}
       +\|\p_s\Lambda^{2-\frac{1}{\ell+1}}\p^1\big(\vpi_{\rho, N}\p^{N-1} w\big)\|^2_0\\
       &\quad+\|\tilde k^{\frac{1}{2}}\p_t\Lambda^{2-\frac{1}{\ell+1}}\p^{1}\big(\vpi_{\rho, N}
       \p^{N-1}w\big)\|^2_{0}+C\norm{\big(\p^{1}\vpi_{\rho, N}\big)\p^{N-1} w}^2_{3-\frac{1}{\ell+1}}\\
       &\leq C_3\set{\norm{\cl \vpi_{\rho, N}\p^{N-1}w}^2_{3-\frac{2}{\ell+1}}+
       \|\vpi_{\rho, N}\p^{N-1} w\|^2_0 +
       \norm{\inner{\p^{1}\vpi_{\rho, N}}\p^{N-1}
       w}^2_{3-\frac{1}{\ell+1}}}.
     \end{align*}
     By the induction assumption, we use now Lemma \ref{lemm2.2},
     to get
    \begin{align*}
    \norm{\cl \vpi_{\rho, N}\p^{N-1} w}_{3-\frac{2}{\ell+1}}&=
    \norm{\cl \vpi_{\rho, N}\p^{N-1} w}_{2+\frac{\ell-1}{\ell+1}}\\
    &\leq
    \frac{ C_2 L^{N-3}}{{\rho}^{(\ell+1)(N-4)}}\inner{\frac{N-1}{\rho}}^{\ell+1}
    \com{\inner{N-4}!}^{\ell+1}\\
    &\leq \frac{2^{\ell+1}C_2 L^{N-3}}{{\rho}^{(\ell+1)(N-3)}}\com{\inner{N-3}!}^{\ell+1}.
  \end{align*}
  Hence the proof will be complete if we can show that (the term $\|\vpi_{\rho, N}\p^{N-1} w\|_0$
  is easier  to treat)
  \begin{equation}\label{2.15}
    \norm{\inner{\p^{1}\vpi_{\rho, N}}\p^{N-1}w}_{3-\frac{1}{\ell+1}}
    \leq\frac{C_4 L^{N-3}}{{\rho}^{(\ell+1)(N-3)}}\com{\inner{N-3}!}^{\ell+1}.
  \end{equation}
Setting $\rho_1=\frac{(N-1)\rho}{N}$, then for any $k\geq2$,
$$
\vpi_{\rho_1, k}=1, \quad \mbox{on} \quad\quad B_{\rho_1},
$$
which implies that $\vpi_{\rho_1, k}=1$ on the Supp $\vpi_{\rho,
N}\subset B_{\rho_1}$ for any $k\geq2$. From \reff{cutoffnorm}, we
have
     \begin{align*}
        &\norm{\inner{\p^{1}\vpi_{\rho, N}}\p^{N-1}w}_{3-\frac{1}{\ell+1}}
        =\norm{\inner{\p^{1}\vpi_{\rho, N}} \vpi_{\rho_1, N-1}\p^{N-1} w}_{2+\frac{\ell}{\ell+1}}\\
        &\leq C_5\set{\inner{\frac{N}{\rho}}\norm{\vpi_{\rho_1, N-1}
        \p^{N-1}w}_{2+\frac{\ell}{\ell+1}}+
        \inner{\frac{N}{\rho}}^{4-\frac{1}{\ell+1}}\norm{\vpi_{\rho_1, N-1}\p^{N-1}w}_{0}}.
     \end{align*}
On the other hand, the induction assumption  with $m=N-1, j=\ell, 0\leq \rho_1\leq 1$, yields
    \begin{align*}
        \frac{N}{\rho}\|\vpi_{\rho_1, N-1}
        \p^{N-1}w\|_{2+\frac{\ell}{\ell+1}}&\leq \frac{N}{\rho}
        \frac{L^{N-3}}{{\rho_1}^{(\ell+1)(N-4)}}\Big(\frac{N-1}{\rho_1}\Big)^\ell
        \com{\inner{N-4}!}^{\ell+1}\\
        &\leq (2 e)^{\ell+1}
        \frac{L^{N-3}}{{\rho}^{(\ell+1)(N-3)}} \com{\inner{N-3}!}^{\ell+1}.
    \end{align*}
Setting now $\tilde{\rho}_1=\frac{(N-2)\rho_1}{N-1}$, then for any
$k\geq 2$,
$$
\vpi_{\tilde{\rho}_1, k}=1, \quad \mbox{on} \quad\quad
B_{\tilde{\rho}_1},
$$
which implies that $\vpi_{\tilde{\rho}_1, k}=1$ on the Supp
$\vpi_{\rho_1, N}\subset B_{\tilde{\rho}_1}$ for any $k\geq 2$. The
induction assumption  with $m=N-3, j=0, 0\leq \tilde{\rho}_1\leq 1$,
yields
    \begin{align*}
       \inner{\frac{N}{\rho}}^{4-\frac{1}{\ell+1}}\norm{\vpi_{\rho_1, N-1}\p^{N-1}w}_{0}
       &=\inner{\frac{N}{\rho}}^{4-\frac{1}{\ell+1}}\norm{\vpi_{\rho_1, N-1}\p^2
       \vpi_{\tilde{\rho}_1, N-3} \p^{N-3}w}_{0}\\
       &\leq \inner{\frac{N}{\rho}}^{4-\frac{1}{\ell+1}}
       \norm{\vpi_{\tilde{\rho}_1, N-3} \p^{N-3}w}_{2}\\
       &\leq\inner{\frac{N}{\rho}}^{4-\frac{1}{\ell+1}}
       \frac{L^{N-5}}{{\tilde{\rho}_1}^{(\ell+1)(N-6)}}\com{\inner{N-6}!}^{\ell+1}\\
        &\leq C_\ell \frac{L^{N-5}}{{\rho}^{(\ell+1)(N-3)}}\com{\inner{N-3}!}^{\ell+1},
    \end{align*}
where we have used the fact that
$$
3(\ell+1)-4+\frac{1}{\ell+1}\geq 0,\quad \forall \,\, \ell\geq 0.
$$
Therefore, we get \reff{2.15} with $C_4=C_5((2 e )^{\ell+1}+2C_\ell$, and finally
for all $0<\rho<1,$
  \begin{align}\label{2.16}
  \begin{split}
    \norm{\vpi_{\rho, N}\p^N w}_{2}
   &+\norm{\p_s\Lambda^{2-\frac{1}{\ell+1}}\vpi_{\rho, N}\p^N w}_{0}
    +\norm{\tilde k^{\frac{1}{2}}\p_t\Lambda^{2-\frac{1}{\ell+1}}\vpi_{\rho, N}\p^N w}_{0}\\
    &\quad\quad\leq \frac{L^{N-2}}{{\rho}^{(\ell+1)(N-3)}}\com{\inner{N-3}!}^{\ell+1},
  \end{split}
  \end{align}
if we choose
$$
L\geq 2 C_3^{1/2}\big(2^{\ell+1} C_2+C_4\big).
$$

\bigbreak \noindent{\bf We prove now that \reff{ind} is true for
{$m=N$ and} $j=j_0+1$  if it is true for {$m=N$ } and $0\leq j \leq
j_0$}.  We apply \reff{08110105} with $r=2+\frac{j_0}{\ell+1}$ and
$v=\vpi_{\rho, N}\p^{N} w\in C^\infty_0(B)$,
\begin{align*}
       &\|\varphi_{\rho, N}\p^N w\|^2_{2+\frac{j_0+1}{\ell+1}}
       +\|\p_s\Lambda^{2+\frac{j_0}{\ell+1}}\vpi_{\rho, N}\p^N w\|^2_{0}
       +\|\tilde k^{\frac{1}{2}}\p_t\Lambda^{2+\frac{j_0}{\ell+1}}\vpi_{\rho, N}\p^N w\|^2_{0}\\
       &\leq C_3\set{\norm{\cl \vpi_{\rho, N}\p^{N}w}^2_{2+\frac{j_0-1}{\ell+1}}+
       \|\vpi_{\rho, N}\p^{N} w\|^2_0}.
\end{align*}
Firstly,
\begin{align*}
\|\vpi_{\rho, N}\p^{N} w\|^2_0\leq \|\vpi_{{\rho_1}, N-2}\p^{N-2}
w\|_2&\leq
\frac{L^{N-4}}{{\rho_1}^{(\ell+1)(N-5)}}\Big(\inner{N-5}!\Big)^{\ell+1}\\
&\leq e^{2(\ell+1)}
\frac{L^{N-4}}{{{\rho}}^{(\ell+1)(N-5)}}\Big(\inner{N-5}!\Big)^{\ell+1}.
\end{align*}

Now for the term $\|\cl \vpi_{\rho, N}\p^{N}w\|_{2+\frac{j_0-1}{\ell+1}}$,
we are exactly in the hypothesis of Lemma \ref{lemm2.1}, \reff{2.12}
implies that
\begin{align*}
       &\|\varphi_{\rho, N}\p^N w\|_{2+\frac{j_0+1}{\ell+1}}
       +\|\p_s\Lambda^{2+\frac{j_0}{\ell+1}}\vpi_{\rho, N}\p^N w\|_{0}
       +\|\tilde k^{\frac{1}{2}}\p_t\Lambda^{2+\frac{j_0}{\ell+1}}\vpi_{\rho, N}\p^N w\|_{0}\\
       &\leq C_3^{1/2}\frac{(C_1+e^{\ell+1}) L^{N-3}}{{\rho}^{(\ell+1)(N-3)}}
    \inner{\frac{N}{\rho}}^{j_0+1}\Big(\inner{N-3}!\Big)^{\ell+1}.
\end{align*}
Finally, if we choose
$$
L\geq \max\left\{M_1,\,2 C_3^{1/2}\big(2^{\ell+1} C_2+C_4\big),\, C_3^{1/2}(C_1+e^{\ell+1}) \right\},
$$
{we get the validity of \reff{ind} for $j=j_0+1,$ and hence for all
$0\leq j\leq \ell+1.$}  Thus the proof  of Proposition \ref{0810272}
is completed.
\end{proof}


\section{Gevrey regularity of solutions for Monge-Amp\`{e}re
equations}\label{0811032}

In this section we prove Theorem \ref{main}.  In the following
discussions, we always assume $u(x,y)$ is a smooth solution of the
 Monge-Amp\`{e}re equation \reff{MA} and $u_{yy}>0$ in
$\Omega,$ a neighborhood of the origin.

We  first introduce the classic partial Legendre transformation (e.g.
\cite{MR1079936}) to translate the Gevrey regularity problem to the
divergence form quasi-linear equation \reff{quasi}. Define the
transformation $T:(x,y)\longrightarrow(s,t)$ by setting
\begin{equation}\label{Len}
  \left\{
   \begin{array}{lll}
      s&=&x,\\
      t&=&u_y.
   \end{array}
  \right.
\end{equation}
It is easy to verify that
\begin{align*}
  J_T=\left(
   \begin{array}{cc}
     s_x&s_y\\
     t_x&t_y
   \end{array}
  \right)=\left(
   \begin{array}{cc}
     1&0\\
     u_{xy}&u_{yy}
   \end{array}
  \right),
  \intertext{and}
  J_T^{-1}=\left(
   \begin{array}{cc}
     x_s&x_t\\
     y_s&y_t
   \end{array}
  \right)=\left(
   \begin{array}{cc}
     1&0\\
     -\frac{u_{xy}}{u_{yy}}&\frac{1}{u_{yy}}
   \end{array}
  \right).
\end{align*}
Thus if $u\in C^\infty$ and $u_{yy}>0$ in $\Omega$, then the
transformations
\[
  T: \Omega\longrightarrow T(\Omega), \quad
  T^{-1}: T(\Omega)\longrightarrow \Omega
\]
are $C^\infty$ diffeomorphism. In \cite{MR1488238}, P. Guan proved
that if $u(x,y)\in C^{1, 1}(\Omega)$ is a weak solution of the
Monge-Amp\`{e}re equation \reff{MA} and $u_{yy}>0$ in $\Omega,$
then $y(s,t)\in C^{0, 1} \inner{T(\Omega)}$ is a weak solution of
equation
  \begin{equation}\label{lin}
    \p_{ss}y+\p_t\Big\{k(s,y(s,t))\p_ty\Big\}=0.
  \end{equation}
He proved also the smoothness of $y(s,t)\in C^\infty(T(\Omega))$
and $u\in C^\infty(\Omega)$.

We prove now the following theorem which, together with Theorem
\ref{08111508}, implies immediately Theorem \ref{main}.

\begin{thm}\label{Gev}
Let $y(s,t)\in G^{\ell+1}\inner{T(\Omega)}$ be a solution of
equation \reff{lin}. Assume that $k(x,y)\in G^{\ell+1}(\Omega).$
Then $u(x,y)\in G^{\ell+1}\inner{\Omega}.$
\end{thm}

We begin with the following results, which can be found in Rodino's
book \cite{Rodino93} (page 21).

\begin{lem}\label{ProGev}
  If $g(z),h(z)\in G^{\ell+1}\inner{U},$  then $(g\,h)(z)\in
  G^{\ell+1}\inner{U},$ and moreover $
  \frac{1}{g(z)}\in G^{\ell+1}\inner{U}$ if $g(z)\neq 0$.
  If $H\in G^{\ell+1}(\Omega)$ and the mapping
  $v: U\longrightarrow\Omega$ is $G^{\ell+1}(U),$
  then $H(v(\cdot))\in G^{\ell+1}(U).$
\end{lem}

We study now the stability of Gevrey regularity by non linear
composition. The following result is due to Friedman
\cite{Friedman58}.
\begin{lem}[Lemma 1 of \cite{Friedman58}]\label{com}
  Let $M_j$ be a sequence of positive numbers satisfying the
  following monotonicity condition:
  \begin{equation}\label{0810181}
    \frac{j!}{i!(j-i)!}M_iM_{j-i}\leq C^*M_j, ~~(i=1,2,\cdots,j;~~
    j=1,2,\cdots)
  \end{equation}
  with $C^*$ a constant. Let $F(z,p)$ be a smooth function defined on
  $\Omega\times(-b,b)\subset\rr^2\times\rr$ satisfying that, for some constant $C,$
  \begin{equation*}
    \max\limits_{(z,p)\in \Omega\times(-b,b)}\abs{\p_z^\gamma\p^i_p F(z,p)}\leq
    C^{\abs\gamma+i}
    M_{\abs\gamma-2}M_{i-2},
  \end{equation*}
for all $\gamma\in\zz_+^2, i\in \zz_+$ with $\abs\gamma, i\geq 2$.
Then there exist two constants $\tilde C,
  C_*,$ depending only on the above constants $C^*$ and $C,$ such that for every
  $H_0,H_1>1$ with $H_1\geq \tilde C H_0,$ if the smooth function $\xi(z)$
  satisfies that $\max\limits_{z\in \Omega}\abs{\xi(z)}<b$ and that
  \begin{align}
     \label{c1}\max\limits_{z\in B}\abs{\p_{z}^\beta\xi(z)}&\leq H_0,\quad
     \textrm{for }~~ \beta ~~\textrm{with}~~ \abs\beta\leq1,\\
     \label{c2}\max\limits_{z\in B}\abs{\p_{z}^\beta\xi(z)}&\leq H_0H_1^{\abs\beta-2}
     M_{\abs\beta-2},
     \quad\textrm{for all}~~ \beta\in\zz_+^2 ~~\textrm{with}~~ 2\leq \abs\beta\leq N,
  \end{align}
  where $N\geq2$ is a given integer,
  then for all $\alpha\in\zz_+^2$ with $\abs\alpha=N,$
  \begin{equation*}
    \max\limits_{z\in B}\abs{\p_{z}^\alpha \inner{F(z,\xi(z))}}\leq
    C_*H_0H_1^{N-2}M_{N-2}.
  \end{equation*}
\end{lem}

\begin{rmk}\label{remark}
Under the same assumptions as the above lemma, if we replace
\reff{c1} and \reff{c2}, respectively, by
\begin{align*}
     \max\limits_{z\in \Omega}\abs{\p_{z_i}^m\xi(z)}&\leq H_0,\quad
     \textrm{for }~~ m\leq 1,\\
     \max\limits_{z\in \Omega}\abs{\p_{z_i}^m\xi(z)}&\leq H_0H_1^{m-2}
     M_{m-2},
     \quad\textrm{for all}~~ m\in\zz_+ ~~\textrm{with}~~ 2\leq m\leq N,
  \end{align*}
  with $1\leq i\leq 2$ some fixed integer and  $N\geq2$ a given integer,
  then
  \begin{equation*}
    \max\limits_{z\in \Omega}\abs{\p_{z_i}^N \inner{F(z,\xi(z))}}\leq
    C_*H_0H_1^{N-2}M_{N-2}.
  \end{equation*}

\end{rmk}

We prepare firstly two propositions. In the follows, let $K$ be any fixed
compact subset of $\Omega.$

\begin{prop}
Assume that $y(s,t)\in G^{\ell+1}(T(\Omega))$ and $k(x, y)\in G^{\ell+1}(\Omega)$, then the
functions $F_m(s,t)\in G^{\ell+1}\inner{T(\Omega)},m=1,2,3,$ where
\begin{align*}
     F_1(s,t)&=(u_{xy}\circ T^{-1})(s,t)=u_{xy}(x(s,t),y(s,t)),\\
     F_2(s,t)&=(u_{xx}\circ T^{-1})(s,t)=u_{xx}(x(s,t),y(s,t)),\\
     \intertext{and}
     F_3(s,t)&=(u_{yy}\circ T^{-1})(s,t)=u_{yy}(x(s,t),y(s,t)).
\end{align*}
\end{prop}

\begin{proof}
  Indeed,  since $y(s,t)\in G^{\ell+1}\inner{T(\Omega)},$ then we conclude $y_s(s,t),y_t(s.t)\in
  G^{\ell+1}\inner{T(\Omega)};$ that is
  \[
   -\frac{u_{xy}(x(s,t),y(s,t))}{u_{yy}(x(s,t),y(s,t))}, ~~
   \frac{1}{u_{yy}(x(s,t),y(s,t))}\in G^{\ell+1}\inner{T(\Omega)}.
  \]
  Lemma \ref{ProGev} yields that $F_3(s,t), F_1(s,t) \in G^{\ell+1}\inner{T(\Omega)}$.
  Moreover, the fact that $k(x,y)\in G^{\ell+1}(\Omega)$ and $x(s,t)=s, y(s,t)\in
  G^{\ell+1}\inner{T(\Omega)},$  implies $k(s,y(s,t))\in
  G^{\ell+1}\inner{T(\Omega)}$, we have, in view of the  equation
  \reff{MA},
  \[
    F_2(s,t)=u_{xx}(x(s,t),y(s,t))\in G^{\ell+1}\inner{T(\Omega)}.
  \]
  This gives the conclusion.
\end{proof}

As a consequence of the above proposition, there exists a constant
$M_*,$ depending only on the Gevrey constants of $k(x,y)$ and
$y(s,t),$ such that for all $i,j\in\zz_+$ with $i,j\geq2,$
  \begin{equation}\label{three}
    \max\limits_{(s,t)\in T(K)}\abs{\p_s^i\p_t^j F_m(s,t)}\leq
    M_*^{i+j}\com{\inner{i-2}!}^{\ell+1}\com{\inner{j-2}!}^{\ell+1},\quad m=1,2,3.
  \end{equation}

\begin{prop}\label{prop3.4}
Assume that $y(s,t)\in G^{\ell+1}(T(\Omega))$ and $k(x, y)\in
G^{\ell+1}(\Omega)$. There exists a constant $\cm,$ depending only
on the Gevrey constants of the functions $y(s,t)$ and $k(x,y),$ such
that for all $i\geq2$,
  \begin{equation}\label{reg-second}
    \max_{(x,y)\in K}\abs{\p_x^i u_y(x,y)}+\max_{(x,y)\in K}\abs{\p_x^i u_x(x,y)}\leq
    2\com{u}_{3,K}\cm^{i-2} \com{(i-2)!}^{\ell+1}.
  \end{equation}
\end{prop}

\begin{proof}
  We first use induction on  integer $i$ to show that
  \begin{equation}\label{x-reg}
    \max\limits_{(x,y)\in K}\abs{\p_x^i u_y(x,y)}\leq \com{u}_{3,K}\cm^{i-2}
    \com{(i-2)!}^{\ell+1},\quad i\geq 2.
  \end{equation}
  Obviously, \reff{x-reg} is  valid for $i=2.$ Now assuming
  \begin{equation}\label{first}
    \max\limits_{(x,y)\in K}\abs{\p_x^i u_y(x,y)}\leq \com{u}_{3,K}\cm^{i-2}
    ((i-2)!)^{\ell+1}, \quad \textrm{for all}~~2\leq i\leq N
  \end{equation}
  with $N\geq 2$ an integer, we need to show that
  \begin{equation}\label{second}
    \max\limits_{(x,y)\in K}\abs{\p_x^{N+1}u_y(x,y)}
    =\max\limits_{(x,y)\in K}\abs{\p_x^{N}u_{xy}(x,y)}\leq \com{u}_{3,K}\cm^{N-1}
    \com{(N-1)!}^{\ell+1}.
  \end{equation}
  Observe that $F_1=u_{xy}\circ T^{-1}$ which implies
  \begin{equation*}
    u_{xy}(x,y)=(F_1\circ T)(x,y)=F_1(x,u_{y}(x,y)).
  \end{equation*}
  Thus the desired estimate \reff{second} will follow if we can
  prove
 \begin{equation}\label{second++}
    \max\limits_{(x,y)\in K}\abs{\p_x^{N}\com{F_1\inner{x,u_y(x,y)}}}\leq \com{u}_{3,K}\cm^{N-1}
    \com{(N-1)!}^{\ell+1}.
  \end{equation}
  In the following we shall apply Remark \ref{remark} to deduce the above
  estimate.

  Define
  \[
    M_j=(j!)^{\ell+1}, \quad H_0=\com u_{3, K},
    \quad H_1=\cm.
  \]
  Clearly $\set{M_j}$ satisfies the  monotonicity condition
 \reff{0810181}. Furthermore, \reff{first} and \reff{three} yield
  \begin{align*}
    \max\limits_{(x,y)\in K}\abs{\p_x^i u_y(x,y)}&\leq H_0,\quad\textrm{for~~}i\leq1,\\
    \max\limits_{(x,y)\in K}\abs{\p_x^i u_y(x,y)}&\leq H_0H_1^{i-2}
    M_{i-2}, \quad \textrm{for all~~}i \textrm{~with~}2\leq i\leq N,\\
    \intertext{and}
    \max\limits_{(s,t)\in T(K)}\abs{\p_s^i\p_t^jF_1(s,t)}&\leq M_*^{i+j}M_{i-2}M_{j-2},
    \quad\textrm{for
    all}~~i,j\in\nn~~\textrm{with}~~i,j\geq2.
  \end{align*}
  Then it follows from the above three inequalities that the conditions
  in Remark \ref{remark} are satisfied, with
  $z_i=x, \xi(z)=u_y(x,y)$ and $F(z,\xi(z))=F_1(x,u_y(x,y)).$
  This yields
  \begin{align*}
    \max\limits_{(x,y)\in K}\abs{\p_x^{N}\com{F_1\inner{x,u_y(x,y)}}}&\leq
    C_*H_0H_1^{N-2}M_{N-2}\\
    &=C_* \com{u}_{3,K}\cm^{N-2}\com{(N-2)!}^{\ell+1}
  \end{align*}
  with $C_*$ a constant depending only on $M_*$ and hence on the
  Gevrey constants of $y(s,t)$ and $k(x,y).$
  Then estimate \reff{second++} follows  if we choose $\cm$ large enough such that $\cm\geq C_*.$
  This completes the proof of \reff{x-reg}.

  Now it remains to prove
  \begin{equation*}
    \max\limits_{(x,y)\in K}\abs{\p_x^i u_x(x,y)}\leq \com{u}_{3,K}\cm^{i-2} \com{(i-2)!}^{\ell+1},
    \quad i\geq 2.
  \end{equation*}
  This can be deduce similarly as above.  In view of \reff{x-reg} and \reff{three},
  we can use Remark \ref{remark}, with
  $z=(x,y), z_i=x, \xi(z)=u_y(x,y)$ and $F(z,\xi(z))= F_2(x,u_y(x,y)),$ to obtain the above estimate.
\end{proof}

\medskip
\noindent \textbf{End of the proof of Theorem \ref{Gev}}:  Now we
can show $u\in G^{\ell+1}(\Omega)$, i.e.,
   \begin{equation}\label{fin}
     \max\limits_{(x,y)\in K}\abs{{\p^\alpha u(x,y)}}\leq 2[u]_{3,K}\cm^{m-3}
     \com{(m-3)!}^{\ell+1},\quad \forall~~\abs\alpha =m\geq3,
  \end{equation}
  where $\cm$ is the constant given in \reff{reg-second}.

  We use induction on $m.$
  The validity of \reff{fin} for $m=3$ is obvious. Assuming, for some positive
  integer $m_0\geq4,$
  \begin{equation}\label{m1}
     \max\limits_{(x,y)\in K}\abs{\p^\gamma u(x,y)}\leq 2[u]_{3,K}\cm^{m-3}
     \com{(m-3)!}^{\ell+1},\quad \forall~~ 3\leq \abs\gamma=m\leq m_0-1.
  \end{equation}
  we need to show the validity of \reff{fin} for $m=m_0.$
  In the following discussions, let $\alpha$ be any fixed multi-index with
  $\abs\alpha=m_0.$ In view of \reff{reg-second}, we only need to
  consider the case when $\p^\alpha=\p^{\tilde\alpha}\p_y^2$ with $\tilde\alpha$
  a multi-index satisfying $\abs{\tilde\alpha}=m_0-2.$
  Observe $F_3=u_{yy}\circ T^{-1}$ which implies
  \[
    u_{yy}(x,y)=\inner{F_3\circ T}(x,y)=F_3(x,u_y(x,y)).
  \]
  Hence
  \[
    \p^\alpha u =\p^{\tilde\alpha}u_{yy}=\p^{\tilde\alpha}\com{F_3(x,u_y(x,y))},
    \quad \abs{\tilde\alpha}=m_0-2.
  \]
  So the validity of  \reff{fin}  for $m=m_0$ will follow if we show
  that, for any  $\abs{\tilde\alpha}=m_0-2,$
  \begin{equation}\label{tildealpha}
    \max\limits_{(x,y)\in K}\abs{\p^{\tilde\alpha}\com{F_3(x,u_y(x,y))}}\leq
    2[u]_{3,K}\cm^{m_0-3}\com{(m_0-3)!}^{\ell+1}.
  \end{equation}
  To obtain the above estimate, we take $M_j,H_0,H_1$ as in the proof of Proposition \ref{prop3.4};
  that is
  \[
    M_j=(j!)^{\ell+1}, \quad H_0=\com u_{3, K},
    \quad H_1=\cm.
  \]
  Then from \reff{three} and the induction assumption \reff{m1},  one has
  \begin{align*}
    \max\limits_{(x,y)\in K}\abs{\p^\gamma u_y(x,y)}&\leq 2H_0,\quad\textrm{for~~}\abs\gamma=m\leq1,\\
    \max\limits_{(x,y)\in K}\abs{\p^\gamma u_y(x,y)}&\leq 2H_0H_1^{m-2}
    M_{m-2}, \quad \textrm{for all~~}2\leq\abs\gamma= m\leq m_0-2,\\
    \intertext{and}
    \max\limits_{(s,t)\in T(K)}\abs{\p_s^i\p_t^j F_3(s,t)}&\leq M_*^{i+j}M_{i-2}M_{j-2}, \quad\textrm{for
    all}~~i,j\in\nn~~\textrm{with}~~i,j\geq2.
  \end{align*}
  Consequently, Lemma \ref{com},  with
  $z=(x,y), \xi(z)=u_y(x,y), N=m_0-2$ and $F(z,\xi(z))=F_3(x,u_y(x,y)),$
  yields for any $\abs{\tilde\alpha}=m_0-2$
  \begin{align*}
     \max\limits_{(x,y)\in K}\abs{\p^{\tilde\alpha}\com{F_3\inner{x,u_y(x,y)}}}&\leq
    2\tilde CH_0H_1^{m_0-4}M_{m_0-4}\\
    &=2\tilde C \com{u}_{3,K}\cm^{m_0-4}\com{(m_0-4)!}^{\ell+1},
  \end{align*}
  where $\tilde C$ is a constant depending only the Gevrey constants
  of $k(x,y)$ and $y(s,t).$  Thus \reff{tildealpha} follows if we choose $\cm$ large enough such that
  $\cm\geq 2\tilde C.$ This gives validity of \reff{fin} for $m=m_0$ and hence for all
  $m\geq 3,$ completing the proof of Theorem \ref{Gev}.


\section{Technical lemmas}\label{lemmas}

In this section, we prove the technical  Lemmas( Lemma \ref{lemm2.1}
and Lemma \ref{lemm2.2}) used in the section \ref{0811031}. Firstly
as an analogue of Lemma \ref{com}, we have
\begin{lem}\label{081101}

Let $N>4$ and $0<\rho<1$ be given. Let $\{M_j\}$ be a positive sequence
satisfying the monotonicity condition
\reff{0810181} and that
\begin{equation*}
   M_{j}\geq \rho^{-j},\quad j\geq 0.
\end{equation*}
Suppose $F(s,t,p),g(s,t)$ are two smooth functions satisfying the following two conditions:

1) There exists a constant $C$ such that for any $j, l\geq 2$,
\begin{equation*}
\big\|\p_{s,t}^{\gamma}\p_p^lF\big\|_{C^{4}(\bar{ B}\times [-b,
b])}\leq  C^{j+l} M_{j-2}M_{l-2},\quad \forall ~\abs\gamma=j,
\end{equation*}
where $b=\com{g}_{0,\bar  B}$ and $\norm{\cdot}_{C^{4}(\bar{
B}\times [-b, b])}$ is the standard H\"{o}rder norm.

2) There exist two constants $H_0,H_1\geq1,$ satisfying $H_1\geq
\tilde C H_0$ with $\tilde C$ a constant depending only on the above
constant $C$,  such that $\com{g}_{6,\bar B}\leq H_0$ and for any
$0<\rho_*<1$ with $\rho_*\thickapprox\rho$ and any $j, 2\leq j\leq
N,$
\begin{equation*}
\|\vpi_{\rho_*,j}\p^j g\|_{\nu}\leq H_0H_1^{j-2}M_{j-2},
\end{equation*}
where  $1<\nu<4$ is a real number.

Then there exists a constant $C_*$ depending only on $C,$ such that
\begin{equation*}
\big\|\vpi_{\rho,N}\p^{N}\inner{F\big(\cdot,
g(\cdot)\big)}\big\|_{\nu}\leq  C_* H_0H_1^{N-2}M_{N-2}.
\end{equation*}
\end{lem}

\begin{proof}
  The proof  is similar to Lemma 5.3 of \cite{CLXK}, so
  we give only main idea of the proof here. In the proof,
  we use $C_n$ to denote  constants which depend only on
$n$ and may be different in different contexts. By Fa\`{a} di Bruno'
formula, $\varphi_{\rho,N} D^\alpha[F(\cdot,g(\cdot))]$ is the
linear combination of terms of the form
\begin{eqnarray}\label{++A}
\varphi_{\rho,N}\inner{\p_{s,t}^\beta\partial_p^lF}(\cdot,g(\cdot))\cdot\prod_{j=1}^l
\p^{\gamma_j}g,
\end{eqnarray}
where $|\beta|+l\leq |\alpha|$ and $
\gamma_1+\gamma_2+\cdots+\gamma_l=\alpha-\beta,$  and  if
$\gamma_i=0$,  $D^{\gamma_i}g$ doesn't appear in (\ref{++A}). Since
$H^{\nu}(\rr^2)$ for $\nu>1$ is an algebra, then we have
\begin{align*}
  &\|\varphi_{\rho,N}\inner{\p_{s,t}^\beta\partial_p^lF}(\cdot,g(\cdot))\cdot\prod_{j=1}^l
\p^{\gamma_j}g\|_{\nu} \\
&\leq\big\|\psi
  \inner{\p_{s,t}^\beta\partial_p^lF}(\cdot,g(\cdot))\big\|_{\nu}
  \cdot\prod_{j=1}^
  {l}
\big\| \varphi_{\rho, \abs{\gamma_j}} \p^{\gamma_j}g\big\|_{\nu},
\end{align*}
where  $\psi\in C_0^\infty(\rr^2)$  and $\psi=1$ on supp
$\varphi_{\rho,N}.$ The above inequality allows us to adopt the
approach used by Friedman  to prove Lemma \ref{com},  to get the
desired estimate. Instead of the $L^\infty$ norm in Lemma \ref{com},
we use $H^\nu$-norm here. But there is no additional difficulty
since $H^\nu(\rr^2)$ is an algebra. We refer to \cite{Friedman58}
for more detail.
\end{proof}

Applying the above result to the functions $\tilde
k(s,t)\stackrel{{\rm def}}{=} k(s,w(s,t))$ and $\tilde
k_w(s,t)\stackrel{{\rm def}}{=}k_y(s,w(s,t)),$ we have

\begin{cor}\label{0811012}
  Let $N_0>4$ and   $j_0\in[0,\ell+1]$ be any given integers. Suppose $k(x,y)\in
  G^{\ell+1}(\rr^2)$ and $w(s,t)\in C^\infty(\bar B)$  satisfying that
  for all  $5\leq m\leq N_0$ and all $\rho$ with $0<\rho<1,$
  \begin{equation}\label{871}
    \norm{\vpi_{\rho, m}\p^m w}_{2+\frac{j_0-1}{\ell+1}}
    \leq
    \frac{c_*L^{m-2}}{{\rho}^{\inner{\ell+1}(m-3)}}
    \inner{\frac{m}{\rho}}^{j_0}\com{\inner{m-3}!}^{\ell+1},
  \end{equation}
   where $L,c_*$ are two constants with $c_*$ independent of $L$. Then there exists a constant
   $\tilde c,$ depending only on the Gevrey constants of $k,w,$ and the above constant $c_*,$
   such that for all $5\leq m\leq N_0$ and all $\rho$ with $0<\rho<1,$
  \begin{align}\label{872}
    &\norm{\vpi_{\rho, m}\p^m \tilde k}_{2+\frac{j_0-1}{\ell+1}}+\norm{\vpi_{\rho, m}
    \p^m \tilde k_w}_{2+\frac{j_0-1}{\ell+1}}\\
    &\qquad\qquad \leq \frac{\tilde c L^{m-2}}{{\rho}^{\inner{\ell+1}(m-3)}}
    \inner{\frac{m}{\rho}}^{j_0}\com{\inner{m-3}!}^{\ell+1}.\nonumber
  \end{align}
\end{cor}

\begin{proof}
 We set
  $H_0=c_*\inner{\com{w}_{8,\bar B}+1}, H_1=L$ and
  \begin{align*}
    M_0=\frac{1}{\rho^3},\quad M_j=\frac{\com{(j-1)!}^{\ell+1}}{\rho^{\inner{\ell+1}(j-1)}}
    \inner{\frac{j+2}{\rho}}^{j_0},\quad j\geq1.
  \end{align*}
  Then by \reff{871}, we have
  \begin{equation*}
    \norm{\vpi_{\rho, m}\p^m w}_{2+\frac{j_0-1}{\ell+1}}
    \leq H_0H_1^{m-2}M_{m-2},\quad  2\leq m\leq N_0.
  \end{equation*}
On the other hand, the fact that $k\in G^{\ell+1}(\rr^2)$, $k_y\in
G^{\ell+1}(\rr^2)$ and $M_j\geq \com{(j-1)!}^{\ell+1}$ implies
  \begin{equation*}
     \big\|\p_{x}^{i}\p_y^j k(x, y) \big\|_{C^{4}(\Omega)}+
     \big\|\p_{x}^{i}\p_y^j k_y(x, y) \big\|_{C^{4}(\Omega)}\leq
     \mathcal {C}^{i+j}  M_{i-2}M_{j-2},\quad
     \forall ~i, j\geq2,
  \end{equation*}
 where $\mathcal {C}$ is the Gevrey constant of $k$.
  Then by Lemma \ref{081101}, the desired inequality \reff{872} will follow if we
  show that $\{M_j\}$ satisfies the
  monotonicity condition \reff{0810181}. For every $0< i< j,$ we
  compute
  \begin{align*}
    {j\choose i}M_iM_{j-i}&=\frac{j!}{i!(j-i)!}
    \frac{\big((i-1)!\big)^{\ell+1}}{\rho^{(\ell+1)(i-1)}}\left(\frac{i+2}{\rho}\right)^{j_0}
    \frac{\Big((j-i-1)!\Big)^{\ell+1}}{\rho^{(\ell+1)(j-i-1)}}\left(\frac{j-i+2}{\rho}\right)^{j_0}\\
    &= \frac{1}{\rho^{(\ell+1)(j-2)}}
    \frac{j!\big((i-1)!\big)^{\ell}\big((j-i-1)!\big)^{\ell}}{i(j-i)}
    \left(\frac{i+2}{\rho}\right)^{j_0}
    \left(\frac{j-i+2}{\rho}\right)^{j_0}\\
    &\leq \frac{9^{j_0}}{\rho^{(\ell+1)(j-2)}}
    \frac{j!\big((j-2)!\big)^{\ell}}{i(j-i)}
    \left(\frac{i}{\rho}\right)^{j_0}
    \left(\frac{j-i}{\rho}\right)^{j_0}\\
    &\leq\left\{9^{\ell+1}\rho^{(\ell+1)-j_0}
    \frac{j^2i^{j_0-1}(j-i)^{j_0-1}}{(j-1)^{\ell+1}(j+2)^{j_0}}\right\}
    \frac{\big((j-1)!\big)^{\ell+1}}{\rho^{(\ell+1)(j-1)}}\left(\frac{j+2}{\rho}\right)^{j_0}\\
    &\leq \left\{9^{\ell+1}\rho^{(\ell+1)-j_0}
    \frac{j^2j^{2(j_0-1)}}{(j-1)^{\ell+1}(j+2)^{j_0}}\right\}
    \frac{\big((j-1)!\big)^{\ell+1}}{\rho^{(\ell+1)(j-1)}}\left(\frac{j+2}{\rho}\right)^{j_0}\\
    &\leq C_{{\ell}}M_j,
  \end{align*}
  where $C_{\ell}$ is a constant depending only on $\ell.$ In the last inequality we used the
  fact that $\ell+1-j_0\geq0.$ This completes the proof of
  Corollary \ref{0811012}.
\end{proof}

\bigbreak We prove now the technical  Lemmas of section
\ref{0811031}. We present  a complete proof of Lemma \ref{lemm2.1},
but omit the proof of Lemma \ref{lemm2.2} since it is similar.

\begin{proof}[{\bf Proof of Lemma \ref{lemm2.1}}]

We recall the hypothesis of Lemma \ref{lemm2.1}; that is, one has

(1) $k\in G^{\ell+1}(\rr^2)$ and $\cl w=0$;

(2) for some $N>5$, \reff{ind} is satisfied for any $5\leq m\leq
N-1$;

(3) for some $0\leq j_0\leq \ell$,
  \begin{align}\label{4.01}
  \begin{split}
    &\norm{\vpi_{\rho, N}\p^N w}_{2+\frac{j_0}{\ell+1}}
    +\norm{\p_s\Lambda^{2+\frac{j_0-1}{\ell+1}}\vpi_{\rho, N}\p^N w}_{0}
    +\norm{\tilde k^{\frac{1}{2}}\p_t\Lambda^{2+\frac{j_0-1}{\ell+1}}\vpi_{\rho, N}\p^N w}_{0}\\
    &\qquad\qquad\leq \frac{C_0L^{N-3}}{{\rho}^{(\ell+1)(N-3)}}
    \inner{\frac{N}{\rho}}^{j_0}\Big(\inner{N-3}!\Big)^{\ell+1}.
  \end{split}
  \end{align}
We want to prove
  \begin{equation}\label{4.02}
    \norm{ \cl \vpi_{\rho, N}\p^N w}_{2+\frac{j_0-1}{\ell+1}}\leq
    \frac{C_1 L^{N-3}}{{\rho}^{(\ell+1)(N-3)}}
    \inner{\frac{N}{\rho}}^{j_0+1}\Big(\inner{N-3}!\Big)^{\ell+1}
  \end{equation}
for all $0<\rho<1$.

It follows from $\cl w=0$   that
  \[\cl \vpi_{\rho,N}\p^\alpha w=\com{\cl,~ \vpi_{\rho,N}}\p^\alpha w
  +\vpi_{\rho,N} \com{\cl,~\p^\alpha}w,\quad \abs\alpha=N.\]
Hence the desired estimate \reff{4.02} will follow if we can prove
that
  \begin{align}\label{4.6}
    \norm{\com{\cl,~ \vpi_{\rho,N}}\p^N w}_{2+\frac{j_0-1}{\ell+1}}&\leq
    \frac{C_1L^{N-3}}{2{\rho}^{(\ell+1)(N-3)}}
    \inner{\frac{N}{\rho}}^{j_0+1}\com{\inner{N-3}!}^{\ell+1},\\
    \intertext{and}
    \label{4.7}
    \sum_{\abs\alpha=N}\norm{\vpi_{\rho,N}\com{\cl,~ \p^\alpha}w}_{2+\frac{j_0-1}{\ell+1}}&\leq
    \frac{C_1L^{N-3}}{2{\rho}^{(\ell+1)(N-3)}}
    \inner{\frac{N}{\rho}}^{j_0+1}\com{\inner{N-3}!}^{\ell+1}.
  \end{align}
We shall proceed to show the above two estimates by the following
steps. As a convention, in the sequel we use $\mathcal {C}_j$ to
denote different constants independent of $L,N.$

  \medskip
  {\bf Step 1.} We claim
  \begin{equation}\label{desired}
    \norm{\vpi_{\rho, m}\p^m w}_{0}
    \leq\frac{\mathcal{C}_1 L^{N-3}}{\rho^{(\ell+1)(N-3)}}
    \inner{\frac{N}{\rho}}^{-2(\ell+1)}\com{(N-3)!}^{\ell+1},
    \quad \forall ~ 3\leq m\leq N.
  \end{equation}

  \smallskip
  To confirm this, we set $\tilde\rho=\frac{(m-1)\rho}{m}.$ Then
  \begin{align*}
    \norm{\vpi_{\rho, m}\p^m w}_{0}
    &=\norm{\vpi_{\rho, m}\p^{2}\vpi_{\tilde\rho,m-2}\p^{m-2}
    u}_{0}\leq\norm{\vpi_{\tilde\rho,m-2}\p^{m-2}
    u}_{2},
  \end{align*}
 we can use \reff{ind} with $j=0$ to compute
  \begin{align*}
    \norm{\vpi_{\tilde\rho, m-2}\p^{m-2}
    u}_{2}&\leq \frac{L^{(m-2)-2}}{{\tilde\rho}^{(\ell+1)((m-2)-3)}}
    \com{\inner{(m-2)-3}!}^{\ell+1}\\
    &\leq \frac{\mathcal{C}_0 L^{N-4}}{{\rho}^{(\ell+1)(m-5)}}
    \com{\inner{m-5}!}^{\ell+1}\\
    &\leq \inner{\frac{N}{\rho}}^{-2(\ell+1)}\frac{\mathcal{C}_0 L^{N-3}}{{\rho}^{(\ell+1)(N-3)}}
    \com{\inner{N-3}!}^{\ell+1},
  \end{align*}
  which implies \reff{desired} at once.

 \medskip
  {\bf Step 2.} In this step, we shall prove the following two
  inequalities:
  \begin{equation}\label{4.11}
    \norm{\inner{\p_{t}\vpi_{\rho, N}}\tilde k\p_t\p^N
    w}_{2+\frac{j_0-1}{\ell+1}}\leq
    \frac{\mathcal{C}_2 L^{N-3}}{\rho^{(\ell+1)(N-3)}}\inner{\frac{N}{\rho}}^{j_0+1}\com{(N-3)!}^{\ell+1}
  \end{equation}
  and
  \begin{equation}\label{4.12}
    \norm{\inner{\p_{s}\vpi_{\rho,N}}\p_s\p^N
    w}_{2+\frac{j_0-1}{\ell+1}}\leq
    \frac{\mathcal{C}_3 L^{N-3}}{\rho^{(\ell+1)(N-3)}}\inner{\frac{N}{\rho}}^{j_0+1}\com{(N-3)!}^{\ell+1}.
  \end{equation}

  \smallskip
  To prove the first inequality \reff{4.11}, we use \reff{cutoffnorm} to get
  \begin{align*}
   &\norm{(\p_t\varphi_{\rho,N})\tilde k\p_{t}\p^N w}_{2+\frac{j_0-1}{\ell+1}}
   =\norm{(\p_t\varphi_{\rho, N})\tilde k\p_{t}\vpi_{\rho_1, N}\p^N
   w}_{2+\frac{j_0-1}{\ell+1}}\\
   &\leq \mathcal{C}_4 \set{\inner{\frac{N}{\rho}}\norm{\tilde k\p_{t}\vpi_{\rho_1, N}\p^N
   w}_{2+\frac{j_0-1}{\ell+1}}
   +\inner{\frac{N}{\rho}}^{3+\frac{j_0-1}{\ell+1}}\norm{\tilde k\p_{t}\vpi_{\rho_1, N}\p^N
   w}_{0}}.
  \end{align*}
  Furthermore, the interpolation inequality \reff{inter} gives
   \begin{align*}
    &\inner{\frac{N}{\rho}}^{3+\frac{j_0-1}{\ell+1}}\norm{\tilde k\p_t\vpi_{\rho_1, N}
    \p^N w}_{0}\\
    &\leq \inner{\frac{N}{\rho}}\norm{\tilde k\p_t\vpi_{\rho_1, N}
    \p^N w}_{2+\frac{j_0-1}{\ell+1}}+\inner{\frac{N}{\rho}}^{4+\frac{j_0-1}{\ell+1}}
    \norm{\tilde k\p_t\vpi_{\rho_1, N}
    \p^N w}_{-1}\\
    &\leq  \inner{\frac{N}{\rho}}\norm{\tilde k\p_t\vpi_{\rho_1, N}
    \p^N w}_{2+\frac{j_0-1}{\ell+1}}+\mathcal{C}_5\inner{\frac{N}{\rho}}^{4+\frac{j_0-1}{\ell+1}}\norm{\vpi_{\rho_1, N}
    \p^N w}_{0}\\
    &\leq  \inner{\frac{N}{\rho}}\norm{\tilde k\p_t\vpi_{\rho_1, N}
    \p^N w}_{2+\frac{j_0-1}{\ell+1}}
    +\frac{\mathcal{C}_6 L^{N-3}}{\rho^{(\ell+1)(N-3)}}\inner{\frac{N}{\rho}}^{j_0+1}\com{(N-3)!}^{\ell+1},
  \end{align*}
where we have used \reff{desired} and $\Lambda^{-1}\tilde k\p_t$ is
bounded in $L^2$. On the other hand,
  note that
  \begin{align*}
   \norm{\tilde k\p_{t}\vpi_{\tilde\rho,N}\p^N
   w}_{2+\frac{j_0-1}{\ell+1}}&\leq \norm{\tilde k\p_{t}\Lambda^{2+\frac{j_0-1}{\ell+1}}\vpi_{\tilde\rho,N}\p^N
   w}_{0}+\norm{\com{\tilde k,~\Lambda^{2+\frac{j_0-1}{\ell+1}}}\p_{t}\vpi_{\tilde\rho,N}\p^N
   w}_{0}\\
   &\leq \mathcal{C}_7\set{\norm{\tilde k^{\frac{1}{2}}\p_{t}\Lambda^{2+\frac{j_0-1}{\ell+1}}\vpi_{\tilde\rho,N}\p^N
   w}_{0}+\norm{\vpi_{\tilde\rho,N}\p^N
   w}_{2+\frac{j_0}{\ell+1}}},
  \end{align*}
  which together with \reff{4.01} yields:
  \begin{align*}
   \norm{\tilde k\p_{t}\vpi_{\tilde\rho,N}\p^N
   w}_{2+\frac{j_0-1}{\ell+1}}
   \leq
   \frac{\mathcal{C}_8 L^{N-3}}{\rho^{(\ell+1)(N-3)}}\inner{\frac{N}{\rho}}^{j_0}\com{(N-3)!}^{\ell+1},
  \end{align*}
  and hence we obtain the desired inequality \reff{4.11}, combining the
  above inequalities. Similar arguments can be applied to prove
  \reff{4.12}. This completes the proof.

  \medskip
  {\bf Step 3.} We now claim that
  \begin{align}\label{4.13}
    \begin{split}
      &\norm{\inner{\p_{ss}\vpi_{\rho, N}}
      \p^N w}_{2+\frac{j_0-1}{\ell+1}}+\norm{\inner{\p_{tt}\vpi_{\rho, N}}
      \tilde k\,\p^N w}_{2+\frac{j_0-1}{\ell+1}}\\
      &\qquad\qquad\leq \frac{\mathcal{C}_9 L^{N-3}}{\rho^{(\ell+1)(N-3)}}
    \inner{\frac{N}{\rho}}^{j_0+1}\com{(N-3)}^{\ell+1}.
    \end{split}
  \end{align}

  \smallskip
  To confirm this,  we use \reff{cutoffnorm} to get
  \begin{align*}
    &\norm{\inner{\p_{ss}\vpi_{\rho,N}}
      \p^N w}_{2+\frac{j_0-1}{\ell+1}}+\norm{\inner{\p_{tt}\vpi_{\rho,N}}
     \tilde k\,\,\p^N w}_{2+\frac{j_0-1}{\ell+1}}\\&
     \leq \mathcal{C}_{10}\set{\inner{\frac{N}{\rho}}^2\norm{ \vpi_{\rho_1, N}
    \p^{N} w}_{2+\frac{j_0-1}{\ell+1}}+\inner{\frac{N}{\rho}}^{4+\frac{j_0-1}{\ell+1}}
    \norm{ \vpi_{\rho_1, N} \p^{N} w}_{0}}.
  \end{align*}
  The interpolation inequality \reff{inter} yields
  \begin{align*}
    \inner{\frac{N}{\rho}}^2\norm{ \vpi_{\rho_1, N}
    \p^{N} w}_{2+\frac{j_0-1}{\ell+1}}&\leq \inner{\frac{N}{\rho}}\norm{ \vpi_{\rho_1, N}
    \p^{N} w}_{2+\frac{j_0}{\ell+1}}\\
    &+\inner{\frac{N}{\rho}}^{2\inner{\ell+1}+j_0+1}\norm{ \vpi_{\rho_1, N}
    \p^{N} w}_{0}.
  \end{align*}
  The  above two inequalities, together with \reff{4.01} and \reff{desired},  give
  the desired estimate \reff{4.13} at once.

  \medskip
  {\bf Step 4.} Now we are ready to prove \reff{4.6},
  the estimate on the commutator of $\cl$ with the cut-off
  function $\vpi_{\rho,N}.$ Firstly, one has
  \begin{align*}
      \com{\cl,~ \vpi_{\rho, N}}=&2\inner{\p_{s}\vpi_{\rho, N}}\p_s
      +\inner{\p_{ss}\vpi_{\rho,N}}+2\inner{\p_{t}\vpi_{\rho,N}}\tilde k\p_t\\
      &\quad+\inner{\p_{tt}\vpi_{\rho,N}}\tilde k+\inner{\p_t\vpi_{\rho,N}}\inner{\p_{t}\tilde
      k}.
  \end{align*}
  Observe that
 \begin{align*}
    \norm{\inner{\p_t\vpi_{\rho, N}}\inner{\p_{t}\tilde
    k}\p^N w}_{2+\frac{j_0-1}{\ell+1}}\leq&
    \mathcal{C}_{11}\Big\{\inner{\frac{N}{\rho}}\norm{ \vpi_{\rho_1,N}
    \p^{N} w}_{2+\frac{j_0-1}{\ell+1}}\\
    &\quad +\inner{\frac{N}{\rho}}^{3+\frac{j_0-1}{\ell+1}}\norm{ \vpi_{\rho_1,N}
    \p^{N} w}_{0}\Big\},
 \end{align*}
  hence from \reff{4.01} and \reff{desired}, we have
  \begin{equation*}
    \norm{\inner{\p_t\vpi_{\rho,N}}\inner{\p_{t}\tilde
      k}\p^N w}_{2+\frac{j_0-1}{\ell+1}}\leq \frac{\mathcal{C}_{12} L^{N-3}}{\rho^{(\ell+1)(N-3)}}
    \inner{\frac{N}{\rho}}^{j_0+1}\com{(N-3)}^{\ell+1}.
  \end{equation*}
  Together with \reff{4.11}, \reff{4.12} and
  \reff{4.13}, this yields the desired estimate
  \reff{4.6} at once.

  \medskip
  {\bf Step 5.} In this step we shall deal with the non linear terms, and prove
  \begin{equation}\label{4.14}
    \norm{\vpi_{\rho, N}\p_t\p^{N}\tilde k
    }_{2+\frac{j_0-1}{\ell+1}}\leq
    \frac{\mathcal{C}_{13} L^{N-3}}{\rho^{(\ell+1)(N-3)}}\inner{\frac{N}{\rho}}^{j_0+1}
    \com{(N-3)!}^{\ell+1}.
 \end{equation}

Recall $\norm{\vpi_{\rho, N}\p_t\p^{N}\tilde k
    }_{2+\frac{j_0-1}{\ell+1}}=\sum_{\abs\alpha=N}\norm{\vpi_{\rho, N}\p_t\p^{\alpha}\tilde k
    }_{2+\frac{j_0-1}{\ell+1}}.$  Leibniz's formula gives, for any $\alpha$ with
$\abs\alpha=N,$
  \begin{align*}
   \vpi_{\rho, N}\p_t\p^\alpha\tilde k&=
    \sum\limits_{1\leq\abs\beta\leq\abs\alpha}{\alpha\choose\beta}\vpi_{\rho,N}\big(\p^\beta\tilde
    k_w\big)\inner{\p_t\p^{\alpha-\beta}w}+\vpi_{\rho,N}\tilde
    k_w\p_t\p^{\alpha}w\\
    &=\sum\limits_{5\leq\abs\beta\leq\abs\alpha-4}{\alpha\choose\beta}\vpi_{\rho,N}\big(\p^\beta\tilde
    k_w\big)\inner{\p_t\p^{\alpha-\beta}w}+\vpi_{\rho,N}\tilde
    k_w\p_t\p^{\alpha}w+R_\alpha
    \end{align*}
    with $\tilde k_w(s,t)=k_w(s,w(s,t)))$ and
    \[R_\alpha=\sum\limits_{1\leq\abs\beta\leq4}{\alpha\choose\beta}\vpi_{\rho,N}\big(\p^\beta\tilde
    k_w\big)\inner{\p_t\p^{\alpha-\beta}w}+\sum\limits_{\abs\alpha-3\leq\abs\beta\leq\abs\alpha}
    {\alpha\choose\beta}\vpi_{\rho,N}\big(\p^\beta\tilde
    k_w\big)\inner{\p_t\p^{\alpha-\beta}w}.\]
Since $H^{\kappa}(\rr^2), \kappa>1$ is an algebra,  we have
  \begin{align*}
    &\sum_{\abs\alpha=N}~\sum\limits_{5\leq\abs\beta\leq
    \abs\alpha-4}{\alpha\choose\beta}\norm{\vpi_{\rho,N}\big(\p^\beta\tilde
    k_w\big)\inner{\p_t\p^{\alpha-\beta}w}}_{2+\frac{j_0-1}{\ell+1}}\\
    &\leq \sum_{\abs\alpha=N}~ \sum\limits_{5\leq\abs\beta\leq \abs\alpha-4}{\alpha\choose\beta}
    \norm{\vpi_{\rho_1,\abs\beta}\p^\beta\tilde k_w}_{2+\frac{j_0-1}{\ell+1}}
    \norm{\vpi_{\rho, N}\p_t\p^{\alpha-\beta}w}_{2+\frac{j_0-1}{\ell+1}}\\
    &\leq \sum\limits_{i=5}^{N-4}\frac{N!}{i!(N-i)!}
    \norm{\vpi_{\rho_1,i}\p^i\tilde k_w}_{2+\frac{j_0-1}{\ell+1}}
    \norm{\vpi_{\rho, N}\p^{N-i+1}w}_{2+\frac{j_0-1}{\ell+1}}.
  \end{align*}
We can use \reff{872} in Corollary \ref{0811012}, to get for each
$i$ with $5\leq i\leq m$
  \begin{align*}
    \|\vpi_{\rho_1,i}\p^i\tilde
    k_w\|_{2+\frac{j_0-1}{\ell+1}}&\leq \frac{\mathcal {C}_{14} L^{i-2}}{{\rho_1}^{\inner{\ell+1}(i-3)}}
    \inner{\frac{i}{\rho_1}}^{j_0}\com{\inner{i-3}!}^{\ell+1}\\
    &\leq \frac{\mathcal {C}_{15} L^{i-2}}{{\rho}^{\inner{\ell+1}(i-3)}}
    \inner{\frac{i}{\rho}}^{j_0}\com{\inner{i-3}!}^{\ell+1}.
  \end{align*}
 Observing  $N-i+1\leq N$ for each $i\geq 1$, we use \reff{desired} and the induction assumptions \reff{ind}  and
  \reff{4.01}, to compute
  \begin{eqnarray*}
   &&\norm{\vpi_{\rho,N}\p^{N-i+1}w}_{2+\frac{j_0-1}{\ell+1}}\leq C
   \Big\{\norm{\vpi_{\rho_1,N-i+1}\p^{N-i+1}w}_{2+\frac{j_0-1}{\ell+1}}\\
   &&\indent\indent+\inner{\frac{N}{\rho}}^{2+\frac{j_0-1}{\ell+1}}\norm{\vpi_{\rho_1,N-i+1}\p^{N-i+1}w}_{0}\Big\}\\
   &&\leq
   \norm{\vpi_{\rho_1,N-i+1}\p^{N-i+1}w}_{2+\frac{j_0-1}{\ell+1}}+\frac{\mathcal {C}_{16} L^{N-i-1}}{\rho^{(\ell+1)(N-i-2)}}
  \inner{\frac{N-i+1}{\rho}}^{j_0}\com{(N-i-2)!}^{\ell+1}\\
  &&\leq
   \norm{\vpi_{\rho_1,N-i+1}\p^{N-i+1}w}_{2+\frac{j_0}{\ell+1}}+\frac{\mathcal {C}_{16} L^{N-i-1}}{\rho^{(\ell+1)(N-i-2)}}
  \inner{\frac{N-i+1}{\rho}}^{j_0}\com{(N-i-2)!}^{\ell+1}\\&&\leq
   \frac{\mathcal {C}_{17} L^{N-i-1}}{\rho^{(\ell+1)(N-i-2)}}
  \inner{\frac{N-i+1}{\rho}}^{j_0}\com{(N-i-2)!}^{\ell+1}.
  \end{eqnarray*}
Then
\begin{align*}
 &\sum_{\abs\alpha=N}~\sum\limits_{5\leq\abs\beta\leq \abs\alpha-4}{\alpha\choose\beta}\|\vpi_{\rho,N}\big(\p^\beta\tilde
    k_w\big)\inner{\p_t\p^{\alpha-\beta}w}\|_{2+\frac{j_0-1}{\ell+1}}\\
 &\leq \sum\limits_{5\leq i\leq N-4}\frac{ N!}{ i!\inner{ N- i}!}
    \frac{\mathcal {C}_{15} L^{ i-2}}{{\rho}^{\inner{\ell+1}( i-3)}}
    \inner{\frac{ i}{\rho}}^{j_0}\com{\inner{ i-3}!}^{\ell+1}\\
    &\qquad\qquad \times\frac{\mathcal {C}_{16} L^{ N- i-1}}{{\rho}^{\inner{\ell+1}( N- i-2)}}
    \inner{\frac{ N- i+1}{\rho}}^{j_0}\com{\inner{ N- i-2}!}^{\ell+1}\\
   &\leq\frac{\mathcal {C}_{18}L^{ N-3}}{{\rho}^{\inner{\ell+1}( N-5)}}
   \inner{\frac{ N}{\rho}}^{2j_0}
    \sum\limits_{5\leq i\leq N-3}\frac{ N!}{ i^3\inner{ N- i}^2}
    \com{\inner{ i-3}!}^{\ell}
    \com{\inner{ N- i-2}!}^{\ell}\\&\leq
    \frac{\mathcal {C}_{18} L^{ N-3}}{{\rho}^{\inner{\ell+1}( N-4)}}\inner{\frac{N}{\rho}}^{j_0+1}
    \sum\limits_{5\leq i\leq N-4}\frac{\inner{ N-5}! N^{5+(j_0-1)}}
    { i^3\inner{ N- i}^2}
    \com{\inner{ N-5}!}^{\ell}\\&\leq
    \frac{\mathcal {C}_{19} L^{ N-3}}{{\rho}^{\inner{\ell+1}( N-3)}}\inner{\frac{N}{\rho}}^{j_0+1}
    \com{\inner{ N-3}!}^{\ell+1}
    \sum\limits_{5\leq i\leq N-4}\frac{ N^{4+j_0+1}}
    { N^{2(\ell+1)} i^3\inner{ N- i}^2}\\&\leq
    \frac{\mathcal {C}_{19}L^{ N-3}}{{\rho}^{\inner{\ell+1}( N-3)}}\inner{\frac{N}{\rho}}^{j_0+1}
    \com{\inner{ N-3}!}^{\ell+1}
    \sum\limits_{5\leq i\leq N-4}\frac{ N^{2}}
    { i^3\inner{ N- i}^2}.
  \end{align*}
   Here the last inequality holds since $4+j_0-2(\ell+1)\leq 2.$  Moreover,
   observing that the series $\sum\limits_{5\leq i\leq N-4}\frac{N^2}{
    i^3(N-i)^2}$ is dominated from above by a constant
    independent of $N$, then
    we get
    \begin{align*}
    &\sum_{\abs\alpha=N}~\sum\limits_{5\leq\abs\beta\leq \abs\alpha-4}{\alpha\choose\beta}\|\vpi_{\rho,N}\big(\p^\beta\tilde
    k_w\big)\inner{\p_t\p^{\alpha-\beta}w}\|_{2+\frac{j_0-1}{\ell+1}}\\
    &\leq
    \frac{\mathcal {C}_{20} L^{ N-3}}{{\rho}^{\inner{\ell+1}( N-3)}}\inner{\frac{N}{\rho}}^{j_0+1}
    \com{\inner{ N-3}!}^{\ell+1}.
  \end{align*}
It's a straightforward verification to prove that
    \begin{align*}
    \sum_{\abs\alpha=N} \|R_\alpha\|_{2+\frac{j_0-1}{\ell+1}}\leq
    \frac{\mathcal {C}_{21} L^{ N-3}}{{\rho}^{\inner{\ell+1}( N-3)}}\inner{\frac{N}{\rho}}^{j_0+1}
    \com{\inner{ N-3}!}^{\ell+1}.
  \end{align*}
So we have proved that
\begin{eqnarray}\label{4.15}
   &&\sum_{\abs\alpha=N}~\sum\limits_{1\leq\abs\beta\leq\abs\alpha}{\alpha\choose\beta}\norm{\vpi_{\rho,N}\big(\p^\beta\tilde
    k_w\big)\inner{\p_t\p^{\alpha-\beta}w}}_{2+\frac{j_0-1}{\ell+1}}\\
    &&\leq
    \frac{\mathcal {C}_{22} L^{ N-3}}{{\rho}^{\inner{\ell+1}( N-3)}}\inner{\frac{N}{\rho}}^{j_0+1}
    \com{\inner{ N-3}!}^{\ell+1}\nonumber.
\end{eqnarray}
 Observe $\norm{\vpi_{\rho,
N}\p_t\p^{N}\tilde k
    }_{2+\frac{j_0-1}{\ell+1}}$ is bounded from above  by
    \begin{align*}
    \sum_{\abs\alpha=N}~\sum\limits_{1\leq\abs\beta\leq\abs\alpha}{\alpha\choose\beta}\norm{\vpi_{\rho,N}\big(\p^\beta\tilde
    k_w\big)\inner{\p_t\p^{\alpha-\beta}w}}_{2+\frac{j_0-1}{\ell+1}}
    +\|\vpi_{\rho,N}\tilde k_w\p_t\p^{N}w
    \|_{2+\frac{j_0-1}{\ell+1}}.
  \end{align*}
So to get the desired estimates \reff{4.14} it remains to estimate
the last term above. Direct calculations yield that
  \begin{align*}
    &\|\vpi_{\rho,N}\tilde k_w\p_t\partial^{N}w
    \|_{2+\frac{j_0-1}{\ell+1}}=\|\vpi_{\rho,N}\tilde k_w\p_t\vpi_{\rho_1, N}\partial^{N}w
    \|_{2+\frac{j_0-1}{\ell+1}}\\
    &\leq \mathcal {C}_{23} \Big\{\|\tilde k_w\Lambda^{2+\frac{j_0-1}{\ell+1}}\p_t\vpi_{\rho_1, N}\partial^{N}w
    \|_{0}+\|\com{\tilde k_w,~\Lambda^{2+\frac{j_0-1}{\ell+1}}}\p_t\vpi_{\rho_1, N}\partial^{N}w
    \|_{0}\\
    &\qquad+\inner{\frac{N}{\rho}}^{2+\frac{j_0-1}{\ell+1}}\|\tilde k_w\p_t\vpi_{\rho_1, N}\partial^{N}w
    \|_{0}\Big\}\\
    &\leq \mathcal {C}_{24} \Big\{\|\tilde k_w\Lambda^{2+\frac{j_0-1}{\ell+1}}\p_t\vpi_{\rho_1, N}\partial^{N}w
    \|_{0}+\|\vpi_{\rho_1, N}\partial^{N}w
    \|_{2+\frac{j_0-1}{\ell+1}}\\
    &\quad\quad\quad
    +\inner{\frac{N}{\rho}}^{2+\frac{j_0-1}{\ell+1}}\|\vpi_{\rho_1, N}\partial^{N}w
    \|_{1}\Big\}.\\
    &\leq \mathcal {C}_{25}\Big\{\|\tilde k_w\Lambda^{2+\frac{j_0-1}{\ell+1}}\p_t\vpi_{\rho_1, N}\partial^{N}w
    \|_{0}+\|\vpi_{\rho_1, N}\partial^{N}w
    \|_{2+\frac{j_0-1}{\ell+1}}\\
    &\quad\quad\quad+\inner{\frac{N}{\rho}}^{\frac{\inner{2\ell+j_0+1}^2}{(\ell+1)(j_0+\ell)}}
    \|\vpi_{\rho_1, N}\partial^{N}w \|_{0}\Big\}.
  \end{align*}
  In the last inequality we used the interpolation inequality  \reff{inter}.
  Combining the fact that
  \begin{align*}
   \abs{  k_w(s,w)}\leq
   C\big(\sup\limits_{w\in\rr}\abs{k_{ww}(s,w)}\big)^{\frac 1 2}
   \inner{k(s,w)}^{\frac 1 2},
  \end{align*}
  which can be deduced from the nonnegativity of $k(s,w)$,  we
  obtain
  \begin{align}
    &\|\vpi_{\rho,N}\tilde k_w(s,w)\p_t\partial^{N}w
    \|_{2+\frac{j_0-1}{\ell+1}}\nonumber\\
    &\leq \mathcal {C}_{26} \Big\{\|\tilde k^{\frac{1}{2}}\Lambda^{2+\frac{j_0-1}{\ell+1}}\p_t\vpi_{\rho_1, N}\partial^{N}w
    \|_{0}+\|\vpi_{\rho_1, N}\partial^{N}w
    \|_{2+\frac{j_0}{\ell+1}}\nonumber\\
    &\quad +\inner{\frac{N}{\rho}}^{\frac{\inner{2\ell+j_0+1}^2}{(\ell+1)(j_0+\ell)}}
    \|\vpi_{\rho_1, N}\partial^{N}w\-_{0}\Big\}\label{4.15+1}\\
    &\leq \frac{\mathcal {C}_{27} L^{\abs\alpha-3}}{\rho^{(\ell+1)(\abs\alpha-3)}}\inner{\frac{N}{\rho}}^{j_0+1}
    \com{(\abs\alpha-3)!}^{\inner{\ell+1}},\nonumber
  \end{align}
  the last inequality following from \reff{4.01} and
  \reff{desired}. The proof is thus completed.

   \medskip
   {\bf Step 6.} Now we prepare to prove the inequality
   \reff{4.7}, the estimate on the commutator of $\cl$ with the differential operator
   $\p^\alpha.$ Direct verification yields
     \begin{align*}
        \com{\cl,~\p^\alpha}w=-\sum_{0<\beta\leq\alpha}{\alpha\choose
        \beta}\inner{\p_t\p^{\beta}\tilde k}\inner{\p_t\p^{\alpha-\beta}w}-\sum_{0<\beta\leq\alpha}{\alpha\choose
        \beta}\inner{\p^{\beta}\tilde k}\inner{\p_{tt}\p^{\alpha-\beta}w}.
     \end{align*}
     So
     \begin{equation}
       \sum_{\abs\alpha=N}\norm{\vpi_{\rho,N}\com{\cl,~\p^\alpha}
       w}_{2+\frac{j_0-1}{\ell+1}}\leq
       \cs_1+\cs_2
     \end{equation}
     with $\cs_1,\cs_2$ given by
     \begin{align*}
        \cs_1&=\sum_{\abs\alpha=N}~\sum_{0<\beta\leq\alpha}{\alpha\choose
        \beta}\norm{\vpi_{\rho,N}\inner{\p_t\p^{\beta}\tilde k}\inner{\p_t\p^{\alpha-\beta}w}
        }_{2+\frac{j_0-1}{\ell+1}}
     \end{align*}
      and
      \begin{align*}
        \cs_2&=\sum_{\abs\alpha=N}~\sum_{0<\beta\leq\alpha}{\alpha\choose
        \beta}\norm{\vpi_{\rho,N}\inner{\p^{\beta}\tilde k}\inner{\p_{tt}\p^{\alpha-\beta}w}
        }_{2+\frac{j_0-1}{\ell+1}}.
      \end{align*}
For $\cs_1$, we have treated the term of $\beta=\alpha$ by
\reff{4.14}, and the terms of $0<\beta<\alpha$ can be deduced
similarly to \reff{4.15}; this gives
      \begin{equation*}
        \cs_{1}\leq
        \frac{\mathcal {C}_{28} L^{N-3}}{\rho^{(\ell+1)(N-3)}}\inner{\frac{N}{\rho}}^{j_0+1}\com{(N-3)!}^{\inner{\ell+1}}.
      \end{equation*}
For $\cs_2$, we have treated the term of $|\beta|=1$ by
\reff{4.15+1}, and the terms of $2\leq|\beta|\leq|\alpha|$ can be
deduced similarly to \reff{4.15}; this gives also
      \begin{equation*}
        \cs_{2}\leq
        \frac{\mathcal {C}_{29} L^{N-3}}{\rho^{(\ell+1)(N-3)}}\inner{\frac{N}{\rho}}^{j_0+1}\com{(N-3)!}^{\inner{\ell+1}}.
      \end{equation*}
   This complete the proof of Lemma \ref{lemm2.1}.
\end{proof}

%

\begin{thebibliography}{10}

\bibitem{MR739925}
L.~Caffarelli, L.~Nirenberg, and J.~Spruck, \emph{The {D}irichlet problem for
  nonlinear second-order elliptic equations. {I}. {M}onge-{A}mp\`ere equation},
  Comm. Pure Appl. Math. \textbf{37} (1984), no.~3, 369--402.

\bibitem{CLXJ}
H.~Chen, W.-X. Li, and C.-J. Xu, \emph{Gevrey hypoellipticity for linear and non-linear Fokker-Planck
equations},  J. Differential Equations \textbf{246} (2009), 320--339.

\bibitem{CLXK}
H.~Chen, W.-X. Li, and C.-J. Xu, \emph{The {G}evrey hypoellipticity for a class
  of kinetic equations}, To appear in ``Comm. Part. Diff. Equ.''

\bibitem{DerridjZuily73-2}
M. Derridj and C. Zuily  \emph{Sur la r\'egularit\'e {G}evrey des
op\'erateurs de {H}\"ormander}, J. Math. Pures Appl. \textbf{52}
(1973), 309--336.

\bibitem{Durand78}
M.~Durand, \emph{R\'{e}gularit\'{e} Gevrey d'une classe
d'op\'{e}rateurs
  hypo-elliptiques}, J. Math. Pures Appl. \textbf{57} (1978), 323--360.

\bibitem{MR730094}
C.~Fefferman and D.H. Phong, \emph{Subelliptic eigenvalue problems}, Conference
  on harmonic analysis in honor of {A}ntoni {Z}ygmund, {V}ol. {I}, {II}
  ({C}hicago, {I}ll., 1981), Wadsworth Math. Ser., Wadsworth, Belmont, CA,
  1983, pp.~590--606.

\bibitem{Friedman58}
A.~Friedman, \emph{On the regularity of the solutions of non-linear elliptic
  and parabolic systems of partial differential equations}, J. Math. Mech.
  \textbf{7} (1958), 43--59.

\bibitem{MR1430436}
P.~Guan, \emph{{$C\sp 2$} a priori estimates for degenerate {M}onge-{A}mp\`ere
  equations}, Duke Math. J. \textbf{86} (1997), no.~2, 323--346.

\bibitem{MR1488238}
P.~Guan, \emph{Regularity of a class of quasilinear degenerate elliptic
  equations}, Adv. Math. \textbf{132} (1997).

\bibitem{MR1687172}
P.~Guan, N.S. Trudinger, and X.-J. Wang, \emph{On the {D}irichlet problem for
  degenerate {M}onge-{A}mp\`ere equations}, Acta Math. \textbf{182} (1999),
  no.~1, 87--104.

\bibitem{guan-sawyer1}
P.~Guan, E. Sawyer \emph{Regularity of subelliptic Monge-Amp\`ere
  equations in the plane}, Trans. AMS. \textbf{361} (2009), 4581-4591.

\bibitem{HZ}
J. X. Hong and C. Zuily, \emph{Existence of $C^\infty$ local
solutions for the {M}onge-{A}mp\`ere equation}, Invent. Math.
\textbf{89}, 645-661 (1987).


\bibitem{MR2137289}
C.~Rios, E.T. Sawyer, and R.L. Wheeden, \emph{A higher-dimensional partial
  {L}egendre transform, and regularity of degenerate {M}onge-{A}mp\`ere
  equations}, Adv. Math. \textbf{193} (2005), no.~2, 373--415.

\bibitem{Rodino93}
L.~Rodino, \emph{Linear partial differential operators in {G}evrey spaces},
  World Scientific Publishing Co. Inc., River Edge, NJ, 1993.

\bibitem{rotschild-stein} L. P. Rotschild and E. M. Stein, \emph{Hypoelliptic differential
operators and nilpotent groups} Acta Math. \textbf{137} (1976), 247-320.

\bibitem{MR1079936}
F.~Schulz, \emph{Regularity theory for quasilinear elliptic systems and
  {M}onge-{A}mp\`ere equations in two dimensions}, Lecture Notes in
  Mathematics, vol. 1445, Springer-Verlag, Berlin, 1990.

\bibitem{Treves80}
F.~Tr{\`e}ves, \emph{Introduction to pseudodifferential and {F}ourier integral
  operators. {V}ol. 1}, Plenum Press, New York, 1980.

\bibitem{MR864418}
C.-J. Xu, \emph{Op\'erateurs sous-elliptiques et r\'egularit\'e des solutions
  d'\'equations aux d\'eriv\'ees partielles non lin\'eaires du second ordre en
  deux variables}, Comm. Partial Differential Equations \textbf{11} (1986),
  no.~14, 1575--1603.

\bibitem{Z}
C. Zuily,  \emph{Sur la r\'egularit\'e des solutions non strictement
convexes de l'\'equations de {M}onge-{A}mp\`ere r\'eelle}, Annali
della Scuola Normale Superiore di Pisa, Classe di Scienze $4^e$
s\'erie, \textbf{11(4)} (1988), 529-554.

\end{thebibliography}

\end{document}